\documentclass[11pt]{amsart}

\usepackage[USenglish]{babel}

\usepackage{amsmath,amsthm,amssymb,amscd}

\usepackage[mathscr]{eucal}
\usepackage[normalem]{ulem}
\usepackage{enumerate}
\usepackage{latexsym,youngtab}
\usepackage{multirow}
\usepackage{epsfig}
\usepackage{parskip}
\usepackage{xy}
\usepackage{tikz}
\usepackage{color}
\usepackage[margin=2cm]{geometry}
\tikzset{
  treenode/.style = {shape=rectangle, rounded corners,
                     draw, align=center,
                     top color=white, bottom color=green!20},
  root/.style     = {treenode, font=\Large, bottom color=blue!30},
  env/.style      = {treenode, font=\ttfamily\normalsize},
  dummy/.style    = {circle,draw}
}

\author[A. Conner, F. Gesmundo, J.M. Landsberg, E. Ventura]{Austin Conner, Fulvio Gesmundo, Joseph M. Landsberg, and Emanuele Ventura}

\address[A. Conner, J.M. Landsberg]{Department of Mathematics, Texas A\&M University, College Station, TX 77843-3368, USA}
\email[A. Conner]{connerad@math.tamu.edu}
\email[J. M. Landsberg]{jml@math.tamu.edu}

\address[F. Gesmundo]{Max Planck Institute for Mathematics in the Sciences, Inselstrasse 22, 04103, Leipzig, Germany}
\email{gesmundo@mis.mpg.de}

\address[E. Ventura]{Universit\`{a} di Torino, Dipartimento di Matematica, via Carlo Alberto 10, 10123 Torino, Italy}
\email{emanuele.ventura@unito.it, emanueleventura.sw@gmail.com}

\title[Tensors with maximal symmetries]{Tensors with maximal symmetries}

\thanks{L. was supported by NSF grants  DMS-1405348 and AF-1814254 as well as a 
Simons Visiting Professor grant supplied by the Simons Foundation and by the 
Mathematisches Forschungsinstitut Oberwolfach. L. and V. were partially 
supported by the grant 346300 for IMPAN from the Simons Foundation and the 
matching 2015-2019 Polish MNiSW fund. G. acknowledges financial support from the 
European Research Council (ERC Grant Agreement no. 337603) and VILLUM FONDEN via 
the QMATH Centre of Excellence (Grant no. 10059).}

\keywords{Symmetry group; Matrix multiplication; Lie algebras; Bilinear forms; Tensor rank}

\subjclass[2010]{15A69, 68Q17, 14L30}

\newcommand{\ot}{\otimes}\newcommand{\ra}{\rightarrow}
\newcommand{\op}{\oplus}

\newcommand{\BC}{\mathbb{C}}

\newcommand{\utriv}{\textit{u-triv}}
\newcommand{\vvirg}{, \dots ,}
\newcommand{\ooplus}{\oplus \cdots \oplus}

\newcommand{\nnn}{\mathbf{n}}

\newcommand{\ep}{\epsilon}
\newcommand{\eps}{\epsilon}

\newcommand{\bbC}{\mathbb{C}}

\newcommand{\bbP}{\mathbb{P}}

\newcommand{\bfn}{\mathbf{n}}
\newcommand{\bft}{\mathbf{t}}
\newcommand{\bfu}{\mathbf{u}}
\newcommand{\bfv}{\mathbf{v}}
\newcommand{\bfw}{\mathbf{w}}

\newcommand{\calB}{\mathcal{B}}

\newcommand{\bfR}{\mathbf{R}}
\newcommand{\ur}{\underline{\mathbf{R}}}
\newcommand{\uR}{\underline{\mathbf{R}}}
\newcommand{\Id}{\mathrm{Id}}
\newcommand{\codim}{\mathrm{codim}}

\newcommand{\Mn}{M_{\langle \bfn \rangle}}

\newcommand{\fraksl}{\mathfrak{sl}}
\newcommand{\frakso}{\mathfrak{so}}
\newcommand{\fsp}{\mathfrak{sp}}
\newcommand{\fraksp}{\mathfrak{sp}}

\newcommand{\fso}{\mathfrak{so}}
\newcommand{\frakS}{\mathfrak{S}}

\newcommand{\frakgl}{\mathfrak{gl}}
\newcommand{\fgl}{\mathfrak{gl}}
\newcommand{\frakg}{\mathfrak{g}}
\newcommand{\fg}{\mathfrak{g}}
\newcommand{\frakh}{\mathfrak{h}}

\newcommand{\frakm}{\mathfrak{m}}

\def\bt{\bold t}

\renewcommand{\bar}[1]{\overline{#1}}

\renewcommand{\hat}[1]{\widehat{#1}}

\newcommand{\rk}{\mathrm{rk}}

\renewcommand{\tilde}[1]{\widetilde{#1}}
\newcommand{\textbinom}[2]{{\textstyle \binom{#1}{#2}}}

 \newtheorem*{theorem*}{Theorem}
\newtheorem{theorem}{Theorem}[section]
\newtheorem{proposition}[theorem]{Proposition}
\newtheorem{lemma}[theorem]{Lemma}
\newtheorem{corollary}[theorem]{Corollary}

\theoremstyle{definition}
\newtheorem{remark}[theorem]{Remark}
\newtheorem{definition}[theorem]{Definition}
\newtheorem{example}[theorem]{Example}
\newtheorem{problem}[theorem]{Problem}

\newcommand{\textfrac}[2]{{\textstyle\frac{#1}{#2}}}
\newcommand{\textsum}{{\textstyle\sum}}

\def\be{\begin{equation}}
\def\ene{\end{equation}}
\def\tmod{\operatorname{mod}}\def\tdim{\operatorname{dim}}
\def\d{\delta}
\def\s{\sigma}

\def\a{\alpha}\def\t{\tau}
\def\hd{,...,}
 
\def\La#1{\Lambda^{#1}}

\def\ol{\overline}

\def\bbta{\mathcal B}

\def\fn{\mathfrak n}

\begin{document}

\begin{abstract} We classify tensors in $\BC^m\ot \BC^m\ot \BC^m$ for $m\geq 7$ with maximal 
and next to maximal dimensional  symmetry groups under a natural genericity 
assumption, called $1$-genericity.  In other words, we classify minimal dimensional  $GL_m^{\times 3}$-orbits in 
$\BC^m\ot \BC^m\ot \BC^m$
assuming $1$-genericity. Our study uncovers new tensors with striking
geometry. This paper was motivated by Strassen's  laser method for bounding
the exponent of matrix multiplication. The best known tensor for the laser 
method
is the  large Coppersmith-Winograd tensor, and our study began with the 
observation
that it has a large symmetry group, of dimension $\binom{m+1}2$. We show that
  in odd dimensions, this is the largest possible dimension for a $1$-generic tensor, but 
in even dimensions
we exhibit a tensor with a larger dimensional symmetry group.  
In the course of the proof, we classify nondegenerate bilinear forms with 
large dimensional stabilizers, which is of interest in its own 
right.   \end{abstract}

\maketitle

\section{Introduction}\label{intro}

Linear algebra has been predominantly
concerned with two-fold tensor products: linear maps from a vector space  $A$ to a vector space $B$ (the space
$A^* \otimes B$),   bilinear forms  (the space $A^* \otimes B^*$), linear endomorphisms  (the space $A^*\ot A$),
and bilinear forms on a single space (the space $A^*\ot A^*$). In  contrast,  remarkably little   is known
about three-fold tensor products, despite their relevance to numerous
important topics  such as  signal processing, complexity theory, classification
of algebras, etc..  See, e.g., \cite{MR2865915} for a discussion.

The orbit structures in $A^*\ot A$ under the action of $GL(A)$ and in $A^*\ot B$ and $A^*\ot B^*$ under
$GL(A)\times GL(B)$ have been known since Jordan.
Already the orbit structure in $A\ot A$ under $GL(A)$ is not completely understood and there is a vast literature
on the subject, see, e.g., \cite{HornSeig:CanonicalFormsComplexMatrixCongruence,MR3958121} and the numerous references therein.
Lemma \ref{solnsizelem} below is an important addition to  the literature, characterizing the small
nondegenerate orbits. More precisely, it determines which 
nondegenerate bilinear forms on $\BC^k$ have stabilizers of dimension at least 
$\binom k2$. 

The corresponding problems for tensors, i.e., orbits in $A\ot B\ot C$ under $GL(A)\times GL(B)\times GL(C)$, is a   wild  problem in the sense of Gabriel's theorem  \cite{MR332887}  (see \cite{MR3727119} for an exposition in English), except in small dimensions and unbalanced cases such as when $\tdim C=\tdim A\tdim B$. In this paper we classify tensors $T\in A\ot B\ot C$ with large symmetry groups, or equivalently small $GL(A)\times GL(B)\times GL(C)$-orbits, in the case $\tdim A=\tdim B=\tdim C$. The smallest orbits in such a tensor space under this action  are classically known. Our study is primarily motivated by the complexity of matrix multiplication, and in this context one imposes a natural genericity  condition on the tensors of interest. This brings into play new small orbits with unexpectedly rich geometric structure.   

Besides their relevance for computer science, our results are   connected 
to a classical question in algebraic
geometry and representation theory: given a representation $V$ of a group $G$, 
what are the vectors $v\in V$ whose orbit closures are of small  dimension, 
i.e.,  with large  stabilizers?
Our main result (Theorem \ref{symthmc}) fits into a long tradition of studying 
small orbits; see for instance 
\cite{kac,MR1251060,MR3185837,MR549013,MR0470096}.

\subsection*{Motivation from matrix multiplication complexity}
The {\it exponent of matrix multiplication}  $\omega$ is a fundamental constant 
that 
controls the complexity of basic operations in linear algebra. 
It is generally conjectured that $\omega = 2$, which would imply that 
one could multiply $\bfn \times \bfn$  matrices using $O(\nnn^{2+\ep})$ 
arithmetic operations for any $\ep>0$. The current state of knowledge is     $2\leq 
\omega < 2.3728596$ \cite{arXiv:2010.05846}, but it has been 
known since 1989 that $\omega \leq 2.3755$ \cite{copwin135}.  

One  motivation for this paper is the {\it Ambainis-Filmus-Le Gall 
challenge}: find new tensors that give good upper bounds on $\omega$ via 
Strassen's laser method \cite{MR882307}. See   \cite{copwin135,BCS,MR3388238} 
for   expositions of the method. This  challenge  is motivated by the results of 
\cite{MR3388238}, where the authors showed that the main tool used so far to 
obtain upper bounds, Strassen's laser method applied to the 
Coppersmith-Winograd 
tensor using coordinate restrictions, can never prove $\omega < 2.3$. Further limitations are proved in 
\cite{DBLP:conf/innovations/AlmanW18,2018arXiv181008671A,2018arXiv181206952C}. Tensors with continuous symmetry are central to the 
implementation of the laser method. 
Advancing ideas in \cite{MR3682743}, we isolate geometric features of 
the Coppersmith-Winograd tensors and find other tensors
with similar features, in the hope they will be useful for the laser method. 
The 
point of departure of this paper was
the observation that Coppersmith-Winograd tensors have very large symmetry 
groups. This led us to the classification problem.
Our main theorem, while uncovering new geometry, fails  to produce 
new tensors good for the laser method, as none of
the tensors in Theorem \ref{symthmc} is better than the big Coppersmith-Winograd 
tensor for 
the laser method. However, in \cite{CGLV2}, guided by the results in this paper, we introduce a skew cousin 
of the small Coppersmith-Winograd
tensor $T_{cw,q}$,   analyze its  utility   for the laser method, 
and show it is  potentially better for 
the laser method
than existing tensors. In particular, $T_{skewcw,2}$, like its cousin 
$T_{cw,2}$, 
potentially could be used to prove $\omega=2$. 
See Corollary  \ref{degen2CW} and the discussion below it for an important consequence for the laser method.

\subsection*{$1$-generic tensors}

A tensor $T \in A \otimes B \otimes C$ is {\it concise} if the induced maps 
$T_A: A^* \to B\ot C$, $T_B: B^*\to A\ot C$, $T_C: C^*\to A\ot B$ are 
injective. 
In our main theorem, we  will require an  additional  natural  genericity condition  
that dates back to 
   \cite{toblerthesis} and
 has been recently studied in 
\cite{MR3578455,MR3682743,ChrVraZui:UniversalPtsAsySpecTensors}. A tensor is 
{\it 
$1_A$-generic}  if the subspace $T_A(A^*)\subset B\ot C$ contains an element of 
maximal rank; $1_A$, $1_B$ or $1_C$-generic tensors are essentially those for 
which Strassen's equations    \cite{Strassen505}  are most useful. A 
tensor is {\it binding} if  it is at least two of  $1_A$, $1_B$, or $1_C$-generic. As observed in 
\cite{MR3578455}, binding tensors are exactly the structure tensors of unital  
(not necessarily associative)  algebras. Binding tensors are 
automatically concise. A tensor is {\it $1$-generic} if it is $1_A,1_B$ and 
$1_C$ generic. ($1$-genericity is called   {\it bequem} in  \cite{toblerthesis} 
and {\it comfortable} in 
 \cite{ChrVraZui:UniversalPtsAsySpecTensors}.)
Propositions \ref{symthma} and \ref{symthmb} respectively determine  the 
maximum 
possible  dimension of the symmetry group of a $1_A$-generic tensor  and a  
binding  tensor and show in each case that  there is a unique such tensor with 
maximal dimensional symmetry. 

Theorem \ref{symthmc} classifies $1$-generic tensors with 
symmetry group of maximal and next to maximal dimension.  In particular, when 
$m$ is even, there is a striking gap in that the second
largest symmetry group has dimension $m-2$ less than the largest.

\subsection*{Notations and conventions} Let $a_1 \vvirg a_{m}$ be a basis of 
the 
vector space $A$, and $\alpha^1 \vvirg \alpha^{m}$  its dual basis of $A^*$. 
Similarly $b_1 \vvirg b_{m}$ and $c_1 \vvirg c_{m}$ are bases of $B$ and $C$ 
respectively, with corresponding dual bases $\beta^1 \vvirg \beta^{m}$ and 
$\gamma^1 \vvirg \gamma^{m}$. Informally, the symmetry group of a tensor $T \in 
A \otimes B \otimes C$ is its 
stabilizer under the natural action
of $GL(A) \times GL(B) \times GL(C)$.
For a tensor $T \in A \otimes B \otimes C$, let $G_T$ denote 
its symmetry group.
 We say $T'$ is {\it isomorphic} to $T$ if  
they are in the same $GL(A) \times GL(B) \times GL(C)$-orbit. We 
identify isomorphic tensors.
Since the action of  $GL(A) \times GL(B) \times GL(C)$ on $A\ot B\ot C$  is not 
faithful, we   work modulo the kernel of its inclusion into $GL(A\ot B\ot C)$, 
which is a $2$-dimensional 
abelian subgroup. See  Section \ref{symgpdef} for details.
The transpose of a matrix $M$ is denoted $M^{\bold t}$. For a set $X$, 
$\overline{X}$ denotes its Zariski closure. For a subset $Y\subset \BC^N$, we 
let $\langle Y\rangle\subset \BC^N$ denote its linear span.

Throughout we
use the summation convention that if a free index appears as a subscript and a superscript, then it is summed over
its range. 
 
\subsection*{Main Theorem} 
\begin{theorem} \label{symthmc} 
Let $m \geq 7$ and let $\dim A = \dim B = \dim C = m$. Let $T \in A \otimes B \otimes C$ be a $1$-generic tensor. Then
\begin{equation}\label{eqn: bound for 1 generic tensor}
 \dim G_T < \frac{m^2}{2} + \frac{m}{2}
\end{equation}
except when $T$ is isomorphic to 
\be\label{Tform} S_{\bbta}:=     a_1 \otimes b_1 \otimes c_m + a_1 \otimes b_m \otimes c_1 + a_m \otimes b_1 \otimes c_1 +  
 \textsum_{\rho =2} ^{m-1} a_1 \otimes b_\rho \otimes c_\rho + \textsum_{\rho = 2}^{m-1} a_\rho \otimes b_1 \otimes c_\rho 
+ \bbta\otimes c_1,
\ene
 where $\bbta\in A\ot B$ is one of the four following bilinear forms

\begin{align}
&\label{skewcwB} \textsum_{\xi=2}^{p+1} a_\xi \otimes b_{\xi + p} - a_{\xi + p } 
\otimes b_\xi \ \ m=2p {\rm \ even}\ \   (T_{skewCW,m-2}) \\
&
\label{CW}\textsum_{\rho=2}^{m-1} a_\rho \otimes b_\rho \ \ {\rm all } \ m  \ \  (T_{CW,m-2}) \\
&
\label{ssCW} a_{m-1} \otimes b_{m-1} + \textsum_{\xi = 2}^{p}  \bigl( a_\xi \otimes b_{\xi + p -1} - a_{\xi + p -1} \otimes b_\eta\bigr) 
 \ \ m=2p {\rm \ even} \ \   (T_{s+skewCW,m-2}) \\
&\label{s+sCW} a_{m-1} \otimes b_{m-1} + \textsum_{\xi = 2}^{p}  \bigl( a_\xi \otimes b_{\xi + p -1} - a_{\xi + p -1} \otimes b_\eta\bigr) 
 \ \ m=2p+1 {\rm \ odd}\ \   (T_{s\oplus skewCW,m-2} )
\end{align}

All these except $T_{skewCW,m-2}$ have $\tdim G_{S_\bbta}=\frac{m^2}{2}+\frac{m}{2}$, and
$\tdim G_{T_{skewCW,m-2}}=\frac{m^2}{2}+\frac {3m}{2}-2$.
\end{theorem}

Theorem \ref{symthmc} implies: 
When $m$ is even, there is a unique, up to isomorphism, $1$-generic tensor $T$ with maximal dimensional
symmetry group, namely $T_{skewCW,m-2}$, and there are 
exactly two, up to isomorphism, additional $1$-generic tensors $T$ such that $\dim G_T \geq \frac{m^2}{2}+\frac{m}{2}$, which are $T_{CW,m-2}$ and $T_{s+skewCW,m-2}$, where equality holds.
When $m$ is odd, there are exactly two  $1$-generic  tensors   $T$ up to isomorphism 
with maximal dimensional symmetry group $\frac{m^2}{2}+\frac{m}{2}$,  which are $T_{CW,m-2}$ and $T_{s\oplus skewCW,m-2}$.

Call a $1$-generic tensor {\it skeletal} if it is of the form \eqref{Tform} for some nondegenerate
bilinear form $\bbta$ (not necessarily one appearing in the theorem).

The Lie algebras of skeletal tensors are given explicitly in Proposition \ref{prop: SB sym}.

The statement of the theorem hints at the method of proof. First, we observe that any $1$-generic tensor may be degenerated to a skeletal tensor with some bilinear form $\bbta$. We then classify bilinear forms with large stabilizer, and show that the corresponding skeletal tensor has symmetry group of dimension lower than $\binom{m+1}{2}$, except for the four cases mentioned in the theorem. Because of the dimensions, we get for free that any tensor that degenerates
to one of them other than $T_{skewCW,m-2}$, i.e.,  \eqref{skewcwB},  satisfies \eqref{eqn: bound for 1 generic tensor}. We then carry out an extensive calculation
to show that any tensor that degenerates to $T_{skewCW,m-2}$ also satisfies 
\eqref{eqn: bound for 1 generic tensor}.   

  \subsection*{Nondegenerate bilinear forms} 
An important step in the proof of Theorem \ref{symthmc} requires us to classify nondegenerate bilinear forms with large symmetry groups. This is done in Lemma \ref{solnsizelem}, where we classify bilinear forms $\calB \in \bbC^k \otimes  \bbC^k$ with large stabilizer under the  action of $GL_k$  given by $g \cdot \bbta = g \bbta g^{\bft}$ for $g \in GL_k$. 

Let $H_{\bbta}\subset GL_k$ be this stabilizer and let $\frakh_{\bbta}$ be its Lie algebra. Elements $X \in \frakh_{\bbta}$ are characterized by the condition 
\begin{equation}\label{eqn: Lie algebra of B}
X \calB + \calB X^{\bold t} = 0.
\end{equation}
The dimension of the solution space of \eqref{eqn: Lie algebra of B} is computed in \cite{TerDop:SolutionsXAplusAXt}, via the Kronecker normal forms of pencil of matrices discussed in \cite{Thomp:PencilsComplexSymSkew,HornSeig:CanonicalFormsComplexMatrixCongruence}. 

In the particular case where the symmetric part or the skew-symmetric part of the bilinear form $\calB$ is non-degenerate, the problem of determining the dimension of $\frakh_\calB$ is equivalent to the one of determining the dimension of certain orbits in the adjoint representations of $\fso_k$ and $\fsp_k$. This is addressed extensively in the literature; see, e.g., \cite{MR0150210,MR432778, MR2734449,MR1832903}.

However, the characterization of \cite{TerDop:SolutionsXAplusAXt}, and the results of the related references, are not suitable for explicitly determining the maximal and next to maximal dimension of $\frakh_\calB$. We characterize these cases explicitly in Lemma \ref{solnsizelem}.

Before stating Lemma \ref{solnsizelem}, we set some notation. Write $W = {\bbC^k}$. As a $\fgl(W)$-module, $W^* \otimes W^* = S^2 W^* \oplus  \Lambda^2 W^*$. Write $\bbta = Q + \Lambda$ with $Q \in S^2 W^*$ symmetric and $\Lambda \in \Lambda^2 W^*$ skew-symmetric. Write $E=\ker(\Lambda)$, $F=\ker (Q)$, which are subspaces of $W$; let $L^*= E^\perp 
\cap F^\perp \subseteq W^*$. Choose a complement $L \subset W$ of $E \oplus F$ so that we have a direct sum decomposition
\[
W= E \oplus L \oplus F,
\]
and we may identify $L$ with the dual space of $L^*$. We may also identify $E^*$ with $(L\op F)^\perp$ and $F^*$ with $(E\op L)^\perp$, regarded as subspaces of $W^*$.

Let $e=\dim E$, $f=\dim F$, and $\ell=\dim L$. Notice that $\rk(\Lambda) = \ell + f$ is even. 

For a subspace $U\subset W$, write $\bbta|_U:= \bbta|_{U\times U}$; this is naturally an element of $(W^* / U^\perp) \otimes (W^* / U^\perp)= U^* \otimes U^*$.  

\begin{lemma}\label{solnsizelem} 
With notations as above, let $\calB \in \bbC^{k*}  \otimes \bbC^{k*}$ be a full rank bilinear form. Then 
\begin{equation}\label{eqn: bound bilinear form}
\dim \frakh_\calB < \binom{k}{2}
\end{equation}
except in the following cases
\begin{enumerate}[(i)]
 \item $(e,\ell,f) = (0,0,k)$ (so $k$ is even): in this case $\calB = 
\Lambda$ is skew-symmetric  and $\frakh_{\bbta} = \fraksp(\Lambda)$ with $\dim 
\frakh_{\bbta} = \binom{k+1}{2}$;
 \item $(e,\ell,f)=(k,0,0)$: in this case $\bbta = Q$ is symmetric 
and $\frakh_{\bbta} = \frakso(Q)$ with $\dim \frakh_{\calB}=\binom{k}{2}$;
 \item  $(e,\ell,f)=(0,1,k-1)$ (so $k$ is even): in this case $\dim 
\frakh_{\bbta}=\binom{k}{2}$;
 \item  $(e,\ell,f)=(1,0,k-1)$ (so $k$ is odd): $\frakh_{\calB} = 
\fraksp( \Lambda \vert_F)$ with $\dim \frakh_{\bbta} = \binom{k}{2}$.
\end{enumerate}
\end{lemma}

\subsection{Overview} In \S\ref{symgpdef} we review basic facts about symmetry groups and algebras
of tensors and their degenerations.
In \S\ref{symproof} we determine the $1_A$-generic  and binding tensors with maximal dimensional
symmetry groups. In \S\ref{biformsub} we prove Lemma \ref{solnsizelem}  on annihilators of bilinear forms.
In \S\ref{bigpfsecta} we begin the proof of Theorem \ref{symthmc}, which we finish in \S\ref{longcalc}.
In \S\ref{brbnds} we show that among skeletal tensors, the only one  with minimal border rank
is the Coppersmith-Wingrad tensor. We observe that this implies a result of \cite{hoyois2021hermitian}
that any $1$-degenerate  minimal border rank tensor degenerates to the Coppersmith-Winograd tensor.
We briefly discuss symmetry algebras of other tensors relevant for the laser method in \S\ref{examplesect}.

\section{The symmetry group of a tensor}\label{symgpdef}

In this section, we define the symmetry group of a tensor and its Lie algebra. 
 
Let $\tilde{\Phi}: GL(A) \times GL(B) \times GL(C)\to GL(A \otimes B \otimes C)$ denote the natural action of $GL(A) \times GL(B) \times GL(C)$ on $A \otimes B \otimes C$. The map $\tilde {\Phi}$ has a two dimensional kernel 
\[
\ker \tilde {\Phi} = \{ (\lambda \Id_A, \mu\Id_B, \nu \Id_C): \lambda \mu \nu = 1\} \simeq 
(\bbC^*)^{\times 2}.
\]
Thus
\begin{equation}\label{Gdef}
G:= \left( GL(A) \times GL(B) \times GL(C) \right) / (\bbC^*)^{\times 2}
 \end{equation}
is naturally a subgroup of $GL(A \otimes B \otimes C)$. 

\begin{definition}
Let $T \in A \otimes B \otimes C$. The symmetry group of $T$, denoted $G_T$, is the stabilizer of $T$ in $G$:
\begin{equation}\label{gtdef}
G_T:=\{ g\in  \left(GL(A)\times GL(B)\times GL(C)\right)/(\BC^*)^{\times 2}  \mid \ g\cdot T=T\}.
\end{equation}
\end{definition}

The symmetry group $G_T$ is a closed algebraic subgroup of $G \subseteq GL(A\otimes B \otimes C)$. If $T $ and $T'$ are isomorphic tensors, then $G_T$ and $G_{T'}$ are conjugate subgroups of $G$; in particular $\dim G_T = \dim G_{T'}$. 

Moreover,   $\dim G_T=\tdim \fg_T$, where $\fg_T$
is the corresponding Lie subalgebra of the algebra $\frakg = \left(\frakgl(A) \oplus \frakgl(B) \oplus \frakgl(C)\right) / \bbC^2$ 
annihilating  $T$:
\[
\frakg_T := \{ L \in \left( \frakgl(A) \oplus \frakgl(B) \oplus \frakgl(C)\right) / \bbC^2 \mid L.T = 0\}.
\]
Here $L.T$ denotes the Lie algebra action.

The algebra $\left(\frakgl(A) \oplus \frakgl(B) \oplus \frakgl(C)\right) / \bbC^2$ is the image of the differential $\Phi = d\tilde{\Phi}$ of the map $\tilde{\Phi}$ defined above, see, e.g., \cite[Section 1.2]{MR2265844}. 
 
It is more convenient to describe the annihilator $\tilde{\frakg}_T = \Phi^{-1}(\frakg_T)$ as a subalgebra of $\frakgl(A) \oplus \frakgl(B) \oplus \frakgl(C)$, acting on $A \otimes B \otimes C$ via the Leibniz rule. Notice that $\tilde{\frakg}_T$ always contains $\ker \Phi = \{ \lambda \Id_A , \mu \Id_A , \nu \Id_A ) : \lambda + \mu + \nu = 0\} \simeq \bbC^2$, and so $\dim G_T = \dim \frakg_T = \dim \tilde{\frakg}_T -2$.

Explicitly, if $L=(U,V,W)\in \fgl(A) \oplus \fgl(B) \oplus \fgl(C)$,
and we write  $U=( u^i_j), V=(v^i_j),W=( w^i_j)$, the condition $L.T = 0$ is equivalent to 
the following    system of linear  equations (we remind the reader of the summation convention):
\begin{equation}\label{killeqn}
u^i_{i'} T^{i' jk}+ v^j_{j'} T^{i j' k} + w^k_{k'} T^{ij k'}=0, \text{ for 
every 
} i,j,k. 
\end{equation}

In what follows, we regard $u^i_{i'} ,v^j_{j'}, w^k_{k'}$ as linear coordinates on $\fgl(A)\op \fgl(B)\op \fgl(C)$, i.e., as basis vectors of the dual space  $\fgl(A)^*\op \fgl(B)^*\op \fgl(C)^*$. There is an inclusion $\tilde \fg_T\subset \fgl(A)\op \fgl(B)\op \fgl(C)$ and the conditions in \eqref{killeqn} are the relations placed on these linear functions when they are pulled back to $\tilde \fg_T$.

In the rest of  the paper, we often display special instances of \eqref{killeqn}, marking them with the corresponding triplet of indices $(ijk)$. 
 
The codimension of $\tilde\frakg_T$ in $\fgl(A) \oplus \fgl(B) \oplus \fgl(C)$ equals the number of linearly independent equations in the system \eqref{killeqn}. In order to prove an upper bound on $\dim \frakg_T$, we will prove lower bounds on the rank of the linear system \eqref{killeqn}. This will be achieved via subsequent normalizations of the tensor $T$, obtained exploiting the genericity hypotheses.

We recall two basic facts, which will be useful throughout the paper. We refer to \cite{FH,MR2265844} for the general theory.

\begin{lemma}\label{lemma: highest weight vector}\label{semiconrem}
 Let $G$ be a reductive algebraic group  and let $W$ be a    $G$-module.
 Let $w,w'\in W$ with $w\in \ol{G\cdot w'}$. Then
 \[
  \dim G_w \geq \dim G_{w'}.
 \]
\end{lemma}

Recall that every nonzero vector contains a highest weight vector in its orbit closure, giving a general upper bound
on the dimensions of stabilizers.

Given $T,T' \in A \otimes B \otimes C$, we say that $T$ \emph{degenerates} to $T'$ if 
$$T'\in \overline{GL(A)\times GL(B)\times GL(C) \cdot T}.
$$
Thus a special case of  Lemma \ref{lemma: highest weight vector} is that if  $T$ degenerates to $T'$, then $\dim G_{T'} \geq \dim G_T$ and the inequality is strict unless $T$ and $T'$ are isomorphic tensors.

\begin{lemma}\label{lemma: annihilator of direct sum}
Let $G$ be a reductive algebraic group and let $W$ be a $G$-module. Let $W = W_1 \ooplus W_k$ be the decomposition of $W$ into its isotypic components. For $w \in W$, write $w = w_1 + \cdots + w_k$ for its isotypic decomposition. Then 
\[
\frakg_w = \frakg_{w_1} \cap \cdots \cap \frakg_{w_k}.
\]
\end{lemma}

The following is an immediate consequence of Theorem \ref{symthmc} and Lemma \ref{semiconrem}.

\begin{corollary}
Let $m \geq 7$ and let $\{1-{\rm generic}\} \subseteq \bbC^m \otimes \bbC^m \otimes \bbC^m$ be the open set of $1$-generic tensors. 

If $m$ is even then 
\[
\overline{GL_m^{\times 3}\cdot T_{skewCW,m-2}}\cap \{1-{\rm generic}\} = GL_m^{\times 3}\cdot T_{skewCW,m-2}.
 \]

If $m$ is odd, then 
\begin{align*}
&\overline{GL_m^{\times 3}\cdot T_{s\oplus skewCW,m-2}}\cap \{1-{\rm generic}\} =GL_m^{\times 3}\cdot T_{s\oplus skewCW,m-2}, \\ 
&\overline{GL_m^{\times 3}\cdot T_{CW,m-2}} \cap \{1-{\rm generic}\}  = GL_m^{\times 3}\cdot T_{CW,m-2}. 
\end{align*}
\end{corollary}

\section{Symmetry groups of tensors: first results}\label{symproof}

In this section, we review the classical result on the largest possible symmetry group of any tensor, and we characterize the maximal possible symmetry group for a $1_A$-generic tensor and for a binding tensor.

\subsection{Arbitrary tensors}

The unique tensor with largest symmetry group in $A \otimes B \otimes C$ is (up to change of bases) $a_1\ot b_1\ot c_1$. This follows immediately from Lemma \ref{lemma: highest weight vector}, since any element of the form $a \otimes b \otimes c$ is highest weight vector in $A \otimes B \otimes C$ for the action of $G=GL(A) \times GL(B)\times GL(C)$, under some choice of   Borel subgroup.

The annihilator of $a_1 \otimes b_1 \otimes c_1$, presented in $(1,m-1)\times (1,m-1)$ block form, is
 \[
\tilde \frakg_{a_1\ot b_1\ot c_1}= \left\{
\begin{pmatrix} u^1_1 & \bfu \\ 0 & \bar{U}\end{pmatrix},
\begin{pmatrix} v^1_1 & \bfv \\ 0 & \bar{V}\end{pmatrix},
\begin{pmatrix} w^1_1 & \bfw \\ 0 & \bar{W}\end{pmatrix}\ \mid
\ \begin{matrix} u^1_1+v^1_1+w^1_1=0, \\ \bfu,\bfv,\bfw\in \BC^{m-1},  
\bar{U},\bar{V},\bar{W}\in \frakgl_{m-1}
\end{matrix}
\right\}.
\]
Hence, $\dim G_{T}= [3(m-1)^2+3(m-1)+2]-2 = 3m^2-3m$. Indeed, the orbit of $a_1 \otimes b_1 \otimes c_1$ under the action of $G$ is the Segre variety of rank one tensors, which has dimension $3m -2 = \dim G - \dim \fg_T$.

\subsection{$1_A$-generic tensors}\label{1Agen}

\begin{proposition}\label{symthma}
Let $T\in A \otimes B \otimes C$ be $1_A$-generic. Then $\dim G_T \leq 2m^2-m-1$ and equality occurs uniquely for the tensor $T_0 =   a_1 \otimes 
(\sum_{j=1}^{m } b_j\ot c_j)$.  
\end{proposition}
\begin{proof}
Let $T\in A\ot B\ot C$  be $1_A$-generic, so there exists $\alpha \in A^*$ such 
that $T_A(\alpha) \in B \otimes C$ has rank $m$. 

After a change of basis in $B$, we may assume that $T_A(\alpha) = 
\sum_{i=1}^m  b_i \otimes c_i$, and after a change of basis in $A$, we may further assume that $\alpha = \alpha^1$. In other words, after a suitable choice of bases in $B$ and in $A$, $T^{1jk} = \delta_{jk}$. 

We first observe that any such tensor degenerates to $T_0$ by the degeneration 
$a_\rho \mapsto 0$ for $\rho = 2 \vvirg m $. Thus, by Lemma \ref{semiconrem} it suffices to
compute $\fg_{T_0}$.

>From \eqref{killeqn}, we have
\begin{equation}\label{oneAeqn}
(1jk) \qquad  u^1_1\delta _{jk}  + v^j_k +w^k_j=0.
\end{equation}
  
Set $2\leq \rho,\sigma ,\tau \leq m$ and use the summation convention. Setting $j=k$, we have 
\begin{equation}\label{oneAeqnb}
(\rho jj) \qquad  u^\rho_1 =0.
\end{equation}

Now \eqref{oneAeqn} shows that the endomorphism $W \in \frakgl(C)$ is completely determined by $u^1_i$ and  $V\in \fgl(B)$. 
In summary, $u^1_i$, $u^\rho_\sigma$ and $V$ completely determine $L \in \tilde{\frakg}_{T_0}$.  Thus   
\[
\tilde \frakg_{T_0}=\left\{
  \begin{pmatrix} -(\mu+\nu) &\bfu \\ 0 & \bar{U} \end{pmatrix}, (\mu\Id+V),
(\nu\Id-V^{\bold t} )
  \mid  \mu,\nu\in \BC, \bar{U} \in \fgl_{m-1},V\in \fgl_m, \bfu \in \BC^{m-1}
\right\}.
\]
\end{proof}

\subsection{Binding tensors}\label{bindsect}


\begin{proposition} \label{symthmb}   Let $T\in A\ot B\ot C=\BC^m\ot \BC^m\ot 
\BC^m$ be a 
binding tensor. Then  $\dim G_T \leq m^2-1$, and equality occurs uniquely
(up to permutation of the three factors)  for the 
tensor 
\[
T_{\utriv,m}:=a_1\ot b_1\ot c_1 + \sum_{\rho=2}^m  a_1\otimes b_\rho \otimes 
c_\rho + \sum_{\rho = 2}^m a_\rho\ot b_1\ot c_\rho .
\]
\end{proposition}
\begin{proof} 
Assume $T$ is $1_A$-generic and $1_B$-generic. As in the proof of Proposition \ref{symthma}, we may assume $T(\alpha^1) \in B \otimes C$ has full rank and normalize it to $\sum_i b_i \otimes c_i$. Note that the action $GL_m \to GL(B) \times GL(C)$ defined by $g \mapsto g \otimes (g^\bft)^{-1}$  preserves this normalization. 

 By assumption, there is $\beta \in B^*$ such that $T(\beta) \in A \otimes C$ has full rank. Using the simultaneous action of $GL_m$ on $B$ and $C$, we may assume $T(\beta^1) \in A \otimes C$ has full rank and normalize it to $T(\beta^1) = \sum_{i} a_i \otimes c_i$, that is $T^{i1k} = \delta_{ik}$.

After this normalization $T = T_{\utriv, m} + T'$ where $T'\in \langle a_2\hd a_m\rangle\ot\langle b_2\hd b_m\rangle\ot C $. Apply the degeneration 
defined by $(X_\ep,Y_\ep,Z_\ep)$ with 
\[
\begin{array}{rllrllrl}
X_\ep : & a_1 \mapsto \textfrac{1}{\ep} a_1 & ~ & Y_\ep: & b_1 \mapsto 
\textfrac{1}{\ep} 
b_1 & ~ & Z_\ep: & c_1 \mapsto \ep^2 c_1 \\
& a_\rho \mapsto \ep a_\rho & ~ & ~& b_\sigma \mapsto \ep b_\sigma & ~ & & 
c_\tau 
\mapsto c_\tau ,
\end{array}
\]
where $\rho,\sigma,\tau \geq 2$.

Among the bases elements appearing in $T$, $a_i \otimes b_j \otimes c_k$ is fixed if and only if $(i,j,k) = (1,1,1)$ or $(i,j,k) = (1,\rho,\rho)$ or 
$(i,j,k) = (\rho,1,\rho)$. All the others basis elements have coefficient $\ep$, $\ep^2$, or $\ep^4$. This shows that $\lim_{\ep \to 0} (X_\ep,Y_\ep,Z_\ep) \cdot T = 
T_{\utriv,m}$, therefore $T$ degenerates to $T_{\utriv,m}$. We conclude that either $T$ and $T_{\utriv,m}$ are isomorphic tensors, or $\dim G_{T} < \dim G_{T_{\utriv.m}}$ by 
Lemma \ref{semiconrem}.

An explicit calculation gives
\begin{align*}
  &\tilde\frakg_{T_{\utriv,m}} = \\
 &\left\{  \left( 
\left(\begin{smallmatrix} \lambda   & \bfu \\ 0 &  -(\mu+\nu)\Id-   
\bar{W}^{\bold t} \end{smallmatrix}\right),
\left(\begin{smallmatrix} \mu  & \bfv \\ 0 & -(\lambda + \nu)\Id - 
\bar{W}^{\bold t}\end{smallmatrix}\right),
\left(\begin{smallmatrix} -\lambda-\mu  & 0 \\ -\bfu^{\bold t} - \bfv^{\bold t} 
&   \nu\Id  + \bar{W} \end{smallmatrix}\right) \right)
 \mid   \begin{array}{l} \lambda,\mu,\nu\in\BC,\\
  \bfu,\bfv \in \BC^{m-1},  \bar{W} \in \fraksl_{m-1} \end{array}  \right\},
\end{align*}
which has dimension $[(m-1)^2-1] + 2 (m-1) + 3$. Hence $\dim 
\frakg_{T_{\utriv,m}} = 
\dim G_{T_{\utriv,m}}  = m^2 -1$. This concludes the proof.
\end{proof}

The normalization performed in the proof of Proposition \ref{symthmb} shows that via the identifications induced by $T(\alpha^1)$ and $T(\beta^1)$, every binding tensor $T$ defines a bilinear map $T: C \times C \ra C$ such that $T(\a^1,\cdot): C\ra C$    
and $T(\cdot,\beta^1): C\ra C$ are both the identity map. The bilinear map $T$ is the structure tensor of a (not necessarily associative) unitary algebra structure on $C$, and $\alpha^1 \simeq \beta^1$ defines the identity element. We refer to \cite{MR3578455}   for a  discussion of this perspective.

\begin{remark} The tensor $T_{\utriv,m}$ is the structure tensor of the trivial unitary algebra of dimension $m$. Explicitly, this algebra may be identified with the quotient $\bbC[x_1 \vvirg x_{m-1}]/\frakm^2$ of the polynomial ring on $m-1$ variables modulo the square of the ideal $\frakm = (x_1 \vvirg x_{m-1})$ generated by the variables.
\end{remark} 

\begin{remark}
The tensor $T_{\utriv,m}$ is concise. It has the largest dimensional symmetry group of any concise tensor we are aware of.
Note that the unique up to scale element of $\La 3\BC^3\subset \BC^3\ot \BC^3\ot \BC^3$ also has
an $m^2-1$ dimensional symmetry group with $m=3$ and this tensor is not $1_A$, $1_B$, or $1_C$-generic.
\end{remark}


\begin{problem}
Determine the largest possible dimension of the symmetry group of a concise tensor. Furthermore, classify concise tensors with symmetry groups of maximal dimension.
\end{problem}

\section{Proof of Lemma \ref{solnsizelem}}\label{biformsub}  

In this section we prove Lemma \ref{solnsizelem}, which classifies bilinear forms
on $\BC^k$  with symmetry group of  dimension at least $\binom k2$.

The annihilator $\frakh_\calB$ of $\calB$ is characterized by equation \eqref{eqn: Lie algebra of B}, which, by Lemma \ref{lemma: annihilator of direct sum}  is equivalent to the two conditions  
\begin{equation}\label{eqn: matrix equation bilinear form}
\begin{aligned}
& X Q + Q X^{\bft} = 0 ,\\ 
& X \Lambda + \Lambda X^{\bft} = 0 .
\end{aligned} 
\end{equation}
where $\calB = Q + \Lambda$ is the decomposition of $\calB$ in its symmetric and skew-symmetric component. 

%

Recall the notation
$W= E \oplus L \oplus F$  with $e=\dim E$, $f=\dim F$, and $\ell=\dim L$.  

If $e = 0$, then $\Lambda$ has full rank, and $\frakh_\calB$ is the annihilator of $Q \in S^2 W^*$ in $\fraksp(\Lambda) \subseteq \frakgl(W)$. In particular $\codim_{\fraksp(\Lambda)} (\frakh_\calB)$ equals the dimension of the $SP(\Lambda)$-orbit of $Q$.

\begin{itemize}
 \item If $\ell = 0$, then $Q=0$ and $f = k$. In this case $\calB = \Lambda$ and $\frakh_\calB = \fraksp(\Lambda)$, with $\dim \frakh_\calB = \binom{k+1}{2}$. This is case (i).

 \item If $\ell = 1$, then $\rk(Q) = 1$. Rank one elements in $S^2 W^*$ are equivalent under the action of $SP(\Lambda)$: the $SP(\Lambda)$-orbit of $Q$ is the affine cone over the Veronese variety $\nu_2(\bbP W^*)$, which 
has dimension $k$. Therefore $\dim \frakh_\bbta = \dim \fraksp(\Lambda) - k = \binom{k+1}{2} -k = \binom{k}{2}$. This is case (iii).
\end{itemize}

The Veronese variety $\nu_2(\bbP W^*)$ is the unique closed $SP(\Lambda)$-orbit in $\bbP S^2 W^*$. As a consequence, if $[Q] \in \bbP S^2 W^*$ is not a point in the Veronese variety, then $\dim (SP(\Lambda) \cdot Q) > \dim \nu_2(\bbP W^*)$ and therefore $\dim \frakh_\calB < \binom{k}{2}$. Thus  if $e=0$  and $\ell>1$, $\bbta$ is eliminated.

If $f = 0$, then $Q$ has full rank and $\frakh_\calB$ is the annihilator of $\Lambda \in \Lambda^2 W^*$ in $\frakso(Q) \subseteq \frakgl(W)$. In particular $\codim _{\frakso(Q)} (\frakh_\calB)$ equals the dimension of the $SO(Q)$-orbit of $\Lambda$.

\begin{itemize}
 \item If $\ell = 0$, then $\Lambda = 0$ and $e = k$. In this case $\calB = Q$ and $\frakh_\calB = \frakso(Q)$, with $\dim \frakso(Q) = \binom{k}{2}$, which  is case (ii).
\end{itemize}
If $\Lambda \neq 0$, then $\dim (SO(Q) \cdot \Lambda) > 0$, therefore $\dim \frakh_\calB < \binom{k}{2}$. This shows that if $f = 0$ and $\ell \geq 1$ then $\calB$ is eliminated.

Now suppose $\ell = 0$ and $e,f >0$. In this case $\frakh_\calB = \frakso(Q|_E) \oplus \fraksp(\Lambda|_F)$, therefore $\dim \frakh_\calB = \binom{e}{2} + \binom{f+1}{2}$. Write $f = k-e$ and consider $\dim \frakh_\calB$ as a function of $e$: 
\[
 \dim \frakh_\calB = \frac{1}{2} \left[e(e-1) + (k-e+1)(k-e)  \right] = \frac{1}{2} [ 2 e^2 - (2+2k)e + (k ^2+k)].
\]
The cases $e = 0$ and $e = k$ were considered above. 
\begin{itemize}
\item If $e = 1$, then $\dim \frakh_\calB = \binom{k}{2}$: this is case (iv).
\end{itemize}

For $e \in \{ 2 \vvirg k-1\}$, then $\dim \frakh_\calB < \binom{k}{2}$: indeed, the maximal value is attained at $e = 3$ and $e =k-1$ and it is $\dim \frakh_\calB = \frac{1}{2} ( k^2 - 3k + 4) = \binom{k-1}{2} + 2 < \binom{k}{2}$ whenever $k > 3$.

Finally, we show that if $e,\ell,f > 0$ then $\dim \frakh_\calB < \textbinom{k}{2}$, eliminating these cases.

Consider  the bilinear form $Q$ restricted to $E\op L$. It  can be fully normalized as 
\begin{equation}\label{eqn: representation Q on E plus L}
Q|_{E\op L}= \left(
\begin{array}{cccc}
\Id_{q} & 0 &0 & 0\\
0& 0 &\Id_{e-q}    & 0\\
0& \Id_{e-q}  &0 & 0\\
0&0  &0 & \Id_{\ell-e+q}
\end{array}\right)
\end{equation}
where $q = \rk(Q|_E : E^* \to E)$.

Moreover, writing $Q,\Lambda$ and $X$ in block form according to the decomposition $E \oplus L \oplus F$, we have
\[
 \calB = Q + \Lambda = \left[\begin{array}{ccc}
                              Q_{EE} & Q_{EL} & 0 \\
                              Q_{EL}^\bft & Q_{LL} & 0 \\
                              0 & 0 & 0
                             \end{array}
 \right] + \left[\begin{array}{ccc}
                              0 & 0 & 0 \\
                              0 & \Lambda_{LL} & \Lambda_{LF} \\
                              0 & -\Lambda_{LF}^\bft & \Lambda_{FF}
                             \end{array} \right]. \qquad  {\rm Write} \ 
X = \left[\begin{array}{ccc}
                              X_{EE} & X_{EL} & X_{EF} \\
                              X_{LE} & X_{LL} & X_{LF} \\
                              X_{FE} & X_{FL} & X_{FF} 
                              \end{array}
 \right] .
\]
Then  \eqref{eqn: matrix equation bilinear form} implies 
\begin{equation*}\label{eqn: blocks ELF sp and so}
\begin{array}{ll}
\left( \begin{smallmatrix} X_{EE} & X_{EL} \\ X_{LE} & X_{LL} \end{smallmatrix} \right) \in \frakso(Q|_{E \oplus L}),   &\left( \begin{smallmatrix} X_{LL} & X_{LF} \\ X_{FL} & X_{FF} \end{smallmatrix} \right) \in \fraksp(\Lambda|_{L \oplus F}), \\
~\\
\left(\begin{smallmatrix} X_{FE} & X_{FL} \end{smallmatrix} \right) = 0, &\left(\begin{smallmatrix} X_{EL} & X_{EF} \end{smallmatrix} \right) = 0 . 
\end{array}
 \end{equation*}
Consider the upper-left size $(e+\ell)\times (e+\ell)$ block of $X$. In blocking $(q,e-q,e-q,\ell-(e-q))$, we have
\[
 \left( \begin{array}{cc} X_{EE} & 0 \\ X_{LE} & X_{LL} \end{array} \right) = \left(\begin{array}{cccc}
  X_{11} & X_{12} & 0 & 0\\
 X_{21} & X_{22} & 0 & 0\\
  X_{31} & X_{32} & X_{33} & X_{34}\\
   X_{41} & X_{42} & X_{43} & X_{44}\end{array}\right).
\]
The condition $XQ + QX^\bft = 0$ provides 
\begin{align*}
\begin{array}{lclcl}
X_{21} = 0, & ~ & X_{31} = -X_{12}^\bft, & ~ & X_{11} \in \frakso_q, \\
X_{41} = 0, & & X_{42} = -X_{34}^\bft, & & X_{32} \in \frakso_{e-q}. \\
 & & X_{22} = -X_{33}^\bft, & & \\
\end{array}
\end{align*}
This gives the upper bound 
\[
 \dim \frakh_\calB \leq \dim \fraksp(\Lambda|_{L \oplus F}) + \underbrace{\dim \frakso_q}_{X_{11}} + \underbrace{\dim \frakso_{e-q}}_{X_{32}} + \underbrace{q \cdot (e-q)}_{X_{12}}.
\]
We conclude
\begin{equation}\label{eqn: bound from normalized Q}
 \dim \frakh_\calB - \binom{k}{2} \leq \binom{\ell + f+1}{2} + \binom{q}{2} + \binom{e-q}{2} +q(e-q) - \binom{k}{2}  = (1-e)(\ell+f).
 \end{equation}
Thus the case $e\geq 2$ is excluded from consideration.

We are left to analyze the case $e = 1$. In this case \eqref{eqn: bound from normalized Q} 
already implies  $\dim \frakh_\calB \leq \binom{k}{2}$ and it remains  to show that the bound is strict. We consider two cases:
 
  If $q =1$ or $q = 0$ and $\ell \geq 2$, then $Q|_L \neq 0$. In this case $\dim \frakh_\calB$ equals the dimension of the annihilator of $Q|_L$ in $\fraksp(\Lambda|_{L \oplus F})$, hence it is strictly smaller than $\dim \fraksp(\Lambda|_{L \oplus F}) = \binom{k}{2}$;

 If $q = 0$ and $\ell = 1$, then one concludes via a direct calculation showing that $\dim \frakh_\calB$ equals the dimension of the annihilator in $\fraksp(\Lambda|_{L \oplus F})$ of a generator of $L$.

This concludes the proof of the Lemma \ref{solnsizelem}. \qed

\section{Proof of Theorem \ref{symthmc}, part one}\label{bigpfsecta}

First, we prove a result which allows us to achieve a convenient normalization of a $1$-generic tensor
(Proposition \ref{prop: basic normalization}), similar to the standard presentation of the Coppersmith-Winograd tensor:
\begin{lemma}\label{lemma: normalization 1-generic}
 Let $T \in A \otimes B \otimes C$ be $1$-generic. Then there exist $\alpha \in A^*, \beta \in B^*, \gamma \in C^*$ such that 
 \begin{itemize}
  \item $T(\alpha,\beta,\gamma) = 0$ 
  \item $T(\alpha,-,-) \in B \otimes C$, $T(-,\beta,-) \in A \otimes C$,
    $T(-,-,\gamma) \in A\otimes B$ are full rank.
 \end{itemize}
\end{lemma}
\begin{proof}
 Let $\Omega_A := \{ \alpha \in A^* : T(\alpha,-,-) \text{ is full rank}\}$ and
 similarly $\Omega_B$, $\Omega_C$; then $\Omega_A,\Omega_B,\Omega_C$ are Zariski
 open in $A^*,B^*,C^*$ respectively. 
 Consider the regular map $T : \Omega_A \times \Omega_B \to C$. It has
 irreducible image which, by conciseness of $T$, is not contained in any
 hyperplane. Consequently, this image is the cone over a positive dimensional set,
 so the set
 \[
\Theta_C = \{ \gamma \in C^*: \gamma^\perp \cap T( \Omega_A \times \Omega_B) \neq \emptyset\}
 \]
 contains a Zariski open set. In particular, $\Theta_C \cap \Omega_C \ne 0$.
 
 Let $\gamma \in \Theta_C \cap \Omega_C$, let $c \in \gamma^\perp \cap T( \Omega_A \times \Omega_B) $ and let $(\alpha,\beta) \in  \Omega_A \times \Omega_B$ be an element satisfying $T(\alpha,\beta) = c$. The triple $(\alpha,\beta,\gamma)$ satisfies the desired conditions.
 \end{proof}

As a consequence of Lemma \ref{lemma: normalization 1-generic}, there exist bases $\{ a_i\}_{i = 1 \vvirg m}$ of $A$, $\{ b_i\}_{i = 1 \vvirg m}$ of $B$, $\{ c_i\}_{i = 1 \vvirg m}$ of $C$, such that 
\[
 T = a_1 \otimes \calB_A + \sigma_{12} ( b_1 \otimes \calB_B) + \calB_C \otimes c_1 +T':
\]
$\s_{12} \in \frakS_{3}$ is the permutation which swaps the first and second factors and 
\begin{itemize}
 \item $\calB_A,\calB_B,\calB_C$ are full rank bilinear forms;
 \item the coefficient of $b_1 \otimes c_1$ in $\calB_A$ is $0$, and similarly for the corresponding coefficient in $\calB_B$ and $\calB_C$;
 \item $T' \in \langle a_2 \vvirg a_m\rangle \otimes \langle b_2 \vvirg b_m\rangle \otimes \langle c_2 \vvirg c_m\rangle $.
\end{itemize}
Let $A' = \langle a_2 \vvirg a_m\rangle$, $A'' = \langle a_2 \vvirg a_{m-1}\rangle$ and similarly on the other factors. 
In this basis consider the coefficient of $b_1\ot c_1$ in $T$; by the above it lies in
$A'$ and is nonzero. Thus we may further change basis in $A'$ so that this part of
$T$ is exactly $a_m\ot b_1\ot c_1$. Doing the same with $B'$ and $C'$, we 
may additionally assume 
\begin{equation*}
\begin{aligned}
  T &= a_1 \otimes b_1 \otimes c_m + a_1 \otimes b_m \otimes c_1 + a_m \otimes b_1 \otimes c_1 \\
& \quad +a_1 \otimes \calB''_A + \sigma_{12} ( b_1 \otimes \calB''_B)
+ \calB''_C \otimes c_1 + \tilde{T}
\end{aligned}
 \end{equation*}
where $\calB''_A \in B'' \otimes C''$, $\calB''_B \in A'' \otimes C''$, $\calB''_C \in A'' \otimes B''$ are full rank bilinear forms, and 
\begin{equation}\label{eqn: tildeT}
\begin{aligned}
\tilde{T} \in A' \otimes B' \otimes C' 
&\op a_1\ot b_m\ot C' \op a_m\ot b_1\ot C' \\
&\op a_1\ot B'\ot c_m\op a_m\ot B'\ot c_1 \\
&\op A'\op b_1\ot c_m\op A'\op b_m\ot c_1.
\end{aligned}
\end{equation}
Now, $\calB_A''$ defines an isomorphism $B'' \simeq (C^*)''$, and similarly $\calB_B''
: A'' \simeq (C^*)''$. Changing bases in $B''$ and $C''$ again, we may suppose
these bases are dual to the distinguished basis in $A''$ under these
isomorphisms. In these bases $\calB_B''$ and $\calB_C''$ are each given by the
identity matrix. Writing $\calB = \calB_A''$, we have shown
\begin{proposition}\label{prop: basic normalization}
  After a change of bases, a $1$-generic tensor $T\in A\ot
  B\ot C$ may be written as $T = S_\calB + \tilde{T}$, where
\begin{equation} \label{eqn: SB}
  \begin{aligned}
    S_\calB &= a_1 \otimes b_1 \otimes c_m + a_1 \otimes b_m \otimes c_1 + a_m \otimes b_1 \otimes c_1 \\
    & \textstyle \quad + \sum_{\rho =2}^{m-1}[a_1 \otimes  b_\rho\ot c_\rho  +
    a_\rho\ot  b_1 \ot c_\rho] + \calB \otimes c_1 
  \end{aligned}
\end{equation}
and $\tilde{T}$ is as in \eqref{eqn: tildeT}.
\end{proposition}

We have called such tensors $S_\calB$ \emph{skeletal}.

\begin{proposition}\label{prop: degeneration to CW from T}
  A tensor $T = S_\calB + \tilde{T}$ normalized as in Proposition \ref{prop: basic normalization} degenerates to $S_\calB$.
\end{proposition}
\begin{proof}
We determine a degeneration which fixes $S_\calB$ and degenerates $\tilde T$ to $0$.
Define $f_\eps \in GL(A)$, $g_\eps \in GL(B)$, $h_\eps \in GL(C)$ as follows:
\[
 \begin{array}{rl}
  f_\eps : & a_1 \mapsto \frac{1}{\eps^2} a_1 \\
  & a_j \mapsto  \eps a_j \quad j=2 \vvirg m-1 \\
  & a_m \mapsto \eps^4 a_m \\
  ~ \\
  g_\eps : & b_1 \mapsto \frac{1}{\eps^2} b_1 \\
  & b_j  \mapsto \eps b_j \quad j=2 \vvirg m-1 \\
  & a_m \mapsto \eps^4 a_m \\
  ~\\
  h_\eps : & c_1 \mapsto \frac{1}{\eps^2} c_1 \\ 
 & c_j  \mapsto \eps c_j \quad j=2 \vvirg m-1 \\
 & c_m \mapsto \eps^4 c_m.
 \end{array}
\]
Then $(f_\eps, g_\eps,h_\eps)(S_\calB) = S_\calB$ and $\lim_{\eps \to 0} (f_\eps
, g_\eps , h_\eps)(\tilde{T}) = 0$, therefore 
\[
 \lim_{\eps \to 0} (f_\eps , g_\eps , h_\eps)(T) = S_\calB. \qedhere
\]
 \end{proof}

It will be convenient to consider $1$-generic tensors normalized as in Proposition 
\ref{prop: basic normalization} more invariantly. In particular, as in the
remarks preceding the proposition, there are natural
identifications $A'' \leftrightarrow B'' \leftrightarrow (C^*)''$. 
Denote these common spaces by $M^*$, so that we have
identifications $A \leftrightarrow L_1^A \op M^* \op L^A$, $B \leftrightarrow
L_1^B \op M^* \op L^B$ and $C \leftrightarrow L_1^C \op M \op L^C$. Then
$S_\calB$ has the expression
\begin{align*}
  S_\calB &= a_1 \otimes b_1 \otimes c_m + a_1 \otimes b_m \otimes c_1 + a_m \otimes b_1 \otimes c_1 \\
& \quad + a_1 \otimes \Id_{M} + \sigma_{12} ( b_1 \otimes \Id_{M}) + \calB \otimes c_1,
\end{align*}
where $\calB \in M^* \ot M^*$. In this presentation it clear that $GL(M) \subset
GL(A) \times GL(B)\times GL(C)$ acts on $S_\calB$ exactly as its action on
bilinear forms $\calB$. More generally, we have 

\begin{proposition}\label{prop: SB sym}
Let $m \geq 5$. Let $\frakh_\calB$ be the annihilator of $\calB$  under the
action of $\frakgl(M^*)$. 
 Then the following is a convenient choice of  lift of $\fg_{S_\bbta}$ to $\fgl(A)\op \fgl(B)\op \fgl(C)$:
\begin{align}\label{LieTb}
&\fg_{S_{\bbta}}=\\
\nonumber   & \left\{\begin{pmatrix}
      2t & \bfv             & u^1_m \\
      0      & X-t \Id_{M^*} & \bfv\calB-\bfu \\
      0      & 0                 & -4t
    \end{pmatrix},
    \begin{pmatrix}
      2t & \bfw             & v^1_m \\
      0      & X-t \Id_{M^*} & \calB \bfw-\bfu \\
      0      & 0                 & -4t    \end{pmatrix},
    \begin{pmatrix}
      2 t & \bfu-\bfv\calB -\calB \bfw & w^1_m \\
      0       & -X^\mathbf{t}-t \Id_M              & -\bfv-\bfw \\
      0       & 0                             & -4t
  \end{pmatrix} \right\}
\end{align}
where  $\bfu \in M^*$, $\bfv,\bfw \in M$, $t\in \BC$, 
$u^1_m+ v^1_m+ w^1_m = 0$, $X \in \frakh_\calB$. In particular, $\dim \frakg_{S_\calB} = 3m -3 + \dim \frakh_\calB$.
\end{proposition}
\begin{proof}
   Fix the index range $\rho,\sigma,\tau = 2 \vvirg m-1$.
Then  \eqref{killeqn} specializes to the following:
\begin{equation*}
  \begin{array}{ll}
(111) & u_m^1 + v_m^1 + w_m^1 = 0 \\
~\\
(11\tau) & u_\tau^1 + v_\tau^1 + w_m^\tau = 0 \\
(1\sigma1) & u_{\sigma'}^1 {\calB}^{\sigma' \sigma} + v_m^\sigma + w_\sigma^1 = 0 \\
(\rho11) & u_m^\rho + v_{\rho'}^1 {\calB}^{\rho \rho'}  + w_\rho^1 = 0 \\
~\\
(1\sigma \tau) & u_1^1 \delta^{\sigma\tau} + v_\tau^\sigma + w_\sigma^\tau = 0 \\
(\rho 1 \tau) & u_\tau^\rho + v_1^1 \delta^{\rho\tau}  + w_\rho^\tau = 0 \\
(\rho\sigma1) & u_{\rho'}^\rho \calB^{\rho'\sigma} + v_{\sigma'}^\sigma \calB^{\rho\sigma'} + w_1^1 \calB^{\rho \sigma} = 0\\
~\\
(\rho\sigma\tau) & u_{1}^{\rho} \delta^{\sigma\tau} + v_1^{\sigma} \delta^{\rho\tau} + w_1^\tau \calB^{\rho \sigma} =0 \\
~\\
(11m) & u_1^1 + v_1^1 + w_m^m = 0 \\
(1m1) & u_1^1 + v_m^m + w_1^1 = 0 \\
(m11) & u_m^m + v_1^1 + w_1^1 = 0 \\
~\\
(1mm) &  v_1^m + w_1^m = 0 \\
(m1m) & u_1^m +  w_1^m = 0 \\
(mm1) & u_1^m + v_1^m = 0 \\
~\\
(1\sigma m) & v_1^\sigma + w_\sigma^m = 0 \\
(m\sigma1) & u_{\rho'}^m \calB^{\rho'\sigma} + v_1^\sigma = 0 \\
(\rho 1 m) & u_1^\rho + w_\rho^m = 0 \\
(\rho m1) & u_1^\rho + v_{\sigma'}^m \calB^{\rho\sigma'}  = 0 \\
(1m\tau) & v_\tau^m + w_1^\tau = 0 \\
(m1\tau) & u_\tau^m + w_1^\tau = 0 \\
  \end{array}
 \end{equation*}
 Equation $(111)$ determines $w^1_m$.
 Write $\ol u=(u^1_2\hd u^1_{m-1})$, $\hat u=(u^2_m\hd u^{m-1}_m)^\bt$ and similarly for $\ol v,\ol w$ and
 $\hat v,\hat w$. The next three equations
 imply
 \begin{align*}
 \hat u&=-\bbta \ol v^\bt-\ol w^\bt\\
 \hat v&=-(\ol u \bbta)^\bt -\ol w^\bt\\
 \hat w&=- \ol u^\bt-\ol v^\bt
 \end{align*}
 Equations $(1\s\t),(\rho 1\t)$ allow us to solve
 for $v^\s_\t,u^\rho_\t$, and plugging the solution into $(\rho\s 1)$ gives
 \begin{align*}
& \bar{W} \calB + \calB \bar{W}^{\bft} = 0\\
&u^1_1+v^1_1+w^1_1=0
\end{align*}
where $\bar{W} = (w_\rho^\sigma)_{\rho,\sigma = 2 \vvirg m-1}$. The
first line  is equivalent to the assertion $\bar{W} \in \frakh_\calB$.

Ignoring the $(\rho\s\t)$ equations for the moment, the next three equations
determine $w^m_m,v^m_m,u^m_m$ as in \eqref{LieTb}, and the next three
form a system that implies $u^m_1,v^m_1,w^m_1=0$.

For every fixed $\rho$, let $\bar{\rho}$ be such that $\bbta^{\rho,\ol\rho}\neq 0$. Then the last six equations form a full rank system in six unknowns, giving 
$u^\rho_1,v^\rho_1,w^{\ol\rho}_1,u^m_{\ol\rho},v^m_{\ol\rho},w^m_\rho=0$.

Finally at this point the equations $(\rho\s\t)$ are trivial and we obtain  \eqref{LieTb} which has
the asserted dimension.
\end{proof}

All of the claims about the exceptional tensors in the theorem follow at once
from Proposition \ref{prop: SB sym}.

It remains to show the upper bound on the dimension of the symmetry group of an
arbitrary $1$-generic tensor $T$. Write $T = S_\calB + \tilde{T}$ as in
Proposition \ref{prop: basic normalization}.
By Lemma \ref{lemma: highest weight vector}, Proposition 
\ref{prop: degeneration to CW from T}, and Proposition \ref{prop: SB sym}, we have
\begin{equation}
  \dim \frakg_T \le \dim \frakg_{S_\calB} = 3m-3+\dim \frakh_\calB,
\end{equation}
with equality if and only if $T \simeq S_{\calB}$. Thus, except in the case
$\calB$ is a nondegenerate skew-symmetric bilinear form, Lemma \ref{solnsizelem}
establishes the remaining claims of the theorem.

We must analyze the case when $\calB = \Lambda$ is a nondegenerate skew-symmetric form. 
In this case $m-2$ is even, so $m = 2p$ is even as well. The analysis above gives
\[
\tdim
\fg_T\leq \tdim \fg_{S_{\Lambda}} = 3m-3+\tdim \fsp(\Lambda)=\frac{m^2}2+\frac {3m}2-2
\]
so it remains to prove that if $\tilde{T}\neq 0$, the dimension drops by at least $m-1$.

\section{End of proof  of Theorem \ref{symthmc}:  the remaining case} \label{longcalc} 
  \def\fu{\mathfrak u}
  \def\ft{\mathfrak t}

First, we obtain a more specific normalization than that shown in Proposition
\ref{prop: basic normalization}. Write $T = S_\calB + \tilde{T}$ as
before, and now change bases such that
\begin{align*}
a_m&\mapsto a_m+ u^1_ma_1\\
b_m&\mapsto b_m+ v^1_mb_1\\
c_m&\mapsto c_m+ w^1_mc_1\\
a_1 &\mapsto a_1+ u^\rho_1a_\rho \\
b_1 &\mapsto b_1+ v^\rho_1b_\rho \\
c_1 &\mapsto b_1+ w^\rho_1c_\rho 
\end{align*}
Then
\begin{align*}
T^{1mm}&\mapsto T^{1mm}+v^1_m+w^1_m\\
T^{m1m}&\mapsto T^{m1m}+u^1_m+w^1_m\\
T^{mm1}&\mapsto T^{mm1}+u^1_m+v^1_m\\
T^{1\rho m}&\mapsto T^{1\rho m}+v^\rho_1 \\
T^{m\rho 1}&\mapsto T^{m\rho 1}+v^\rho_1 \\
T^{1m\rho}&\mapsto T^{1m\rho}+w^\rho_1 \\
T^{m1\rho}&\mapsto T^{m1\rho}+w^\rho_1 \\
T^{ \rho 1m}&\mapsto T^{\rho 1m}+u^\rho_1 \\
T^{ \rho m1}&\mapsto T^{ \rho m1}+u^\rho_1 
\end{align*}
and the only other entries of $T$ which change are $T^{ijk}$, where $i,j,k \ge
2$. In particular, after such a change of variables, $T$ is still of the form
$S_\calB + \tilde{T}$.
Choose $u^1_m,v^1_m,w^1_m$ such that
$-T^{1mm}=v^1_m+w^1_m$, 
$-T^{m1m}=u^1_m+w^1_m$, and 
$- T^{mm1}=u^1_m+v^1_m$, to send these quantities to zero.
Choose $u^\rho_1=-T^{\rho 1m}$, $v^\rho_1=-T^{1\rho m}$ and $w^\rho_1=-T^{1m\rho}$
to send these quantities to zero. We have shown that an arbitrary $1$-generic
tensor may be written as $S_\calB + \tilde T$, where now 
\begin{equation*}
\tilde{T} \in A' \otimes B' \otimes C' 
\op a_m\ot b_1\ot C'' 
\op a_m\ot B''\ot c_1 
\op A''\op b_m\ot c_1.
\end{equation*}

 In the skeletal case we have $\fu:=\langle u^1_1,u^1_\rho,u^1_m,v^1_1,v^1_\s,v^1_m,w^1_1, w^1_\t\rangle $
 (which is of dimension $3m-1$) independent and
 independent of elements of $\fsp(M)$. We need to show that after our
 normalizations,  if $\tilde{T} \neq 0$, then there are at least
 $m-1$ relations among the elements of $\fu,\fsp(M)$. 

Fix index ranges
$2\leq \xi,\eta\leq p$,  $2\leq \rho,\s,\t\leq m-1$. Write
$\fn:=\langle u^1_\rho, v^1_\s, w^1_\t, u^1_m,v^1_m\rangle$,
the nilpotent part of $\fu$.

We have
$T^{111},T^{11\t},T^{1\s 1},T^{\rho 11},T^{mm*},T^{m*m},T^{*mm} =0$, $T^{11m}=T^{1m1}=T^{m11}=1$,
and by the normalization lemma $T^{1m\rho}=T^{1 \rho m}=T^{\rho 1m}=0$.

  Let $\ep^\rho=1$ if $\rho\leq p$ and $\ep^\rho=-1$ if $\rho>p$. 
Let $\bar\rho=\rho+p$ if $\rho\leq p$ and $\bar\rho=\rho-p$ if $\rho>p$.
We have $T^{\rho\s 1}=\ep^\rho\d^{\s\ol\rho}$.  Write $A=L_1^A\op M^A\op L_m^A$ etc...

\subsection{Outline} ~ \\ 
Step 1: Solve
$u^\s_m,v^\s_m,w^\s_m,w^1_m$ (i.e., elements of $L_m^A\ot (M^A)^*, L_m^B\ot (M^B)^*,L_m^C\ot (M^C)^*, L_m^C\ot (L_1^C)^*$) in terms of $\fu$, solve  
$u^\rho_\s,v^\rho_s$ (i.e., elements of $\fgl(M)_A,\fgl(M)_B$) in terms of $\fu\op \fsp(M)_C\op \fsp(M)_C^c$,
and  solve $u^m_m,v^m_m,w^m_m$ in terms of $\fu$.

Step 2: Using $(\rho\s 1)$ solve $\fsp(M)_C^c$ in terms of $\fu$.

Step 3: use $(\rho\s\t)$ to severely restrict $T^{\rho\s\t}$. Namely
among the irreducible modules in $M^{\ot 3}$ only the highest weight vector of $S^3M$ and the three copies
of $M$ (which we denote $M_A^s,M_B^s,M_C^s$ can occur, i.e., we have $3(m-2)+1$ parameters instead of $(m-2)^3$.
Moreover, if a highest weight vector is nonzero, at most  one more relation
among elements of $\fu\op \fsp(M)_C$ is allowed.

Step 4:  
We observe that the $\fsp(M)\op \fsp(M)^c$ terms   do not appear if we symmetrize the $(\rho\s 1)$ equations.
In the special case $\rho=\s=2$, we obtain in particular
$u^1_2T^{22(p+2)}\equiv 0\tmod$ other basis elements of $ \fu$, which shows that if $T^{22(p+2)}\neq 0$,
no more relations among elements of $\fu\op \fsp(M)_C$ are allowed.
Assuming $T^{22(p+2)}\neq 0$, we  obtain three relations among six modules isomorphic to $M$ that must
occur if no further relations among the $\fu$ are to occur. Explicitly the relations are 
   $L_m^A\ot M^B\ot L_1^C=T^{m\s 1}=0$, $M_B^s=t^{ \s}_B= \ep^\s t^{\ol\s}_C=M_C^s$,
$M^A\ot L_m\ot L_1=t^\s_A=\ep^{\s}T^{m1{\ol\s}}  +T^{m\s 1} -  \ep^\s t^{\ol\s}_C=L_m^A\ot L_1^B\ot M^C+
L_m^A+L_m^A\ot M^B\ot L_1^C+ M_C^s$.
 
 Step 5: Using $(\rho\s\t)=(22(p+2))$ modulo $\fu$  we obtain another relation which shows
 $T^{22(p+2)}=0$.
 
 Step 6: We revisit the symmetrized $(\rho\s 1)$ equations with the knowledge $T^{22(p+2)}=0$
 and obtain $3$ expressions involving the six modules
 $T^{m\s 1}$ (i.e. $M^A\ot L^B_m\ot L^C_1$), $T^{m1\s}$, $T^{\s m1}$ (recall
 that its other permutations have been normalized to zero), and $t_A^\rho,t^\s_B,t^\t_C$
 which are defined to be the three copies of $M$ in $M^{\ot 3}$, and if any of these
 expressions is nonzero, we obtain $m-2$ relations and if any two are nonzero we obtain more than
 $m-1$ relations so at most one of the expressions may be nonzero.
 
 Step 7: We explicitly solve for $u^\s_1,v^\s_1,w^{\ol \s}_1,u^m_{\ol\s}, v^m_{\ol\s},w^m_\s$
  in terms of elements of $\fn\op \fsp(M)\op \fsp(M)^c$.
  
  Step 8: We consider the $(\rho\s\t)$ equations modulo $\fu$ and obtain $6$ additional expressions.  
  among the $6$ quantities of Step 6, and observe that at most two  of these $6$ expressions can be nonzero,
  as each imposes $\frac{m-2}2$ conditions. Even if two of the expressions are nonzero, we have
  enough equations combined with the previous to show all the terms defined in Step 6 are zero.

 Step 9: We show $T^{\rho\s m},T^{\rho m\t}, T^{m\s\t}$, each of which lie in $M\ot M\sim M^*\ot M$ (tensored with a line) are zero by first
 considering the $(\rho\s m)$ equations and their permutations modulo $\fn$
 to show  they can only be the trivial representation plus the highest weight vector in
 $S^2M$ and if the highest weight vector in $S^2M$ is nonzero, there can be at most
 two more relations. Here we reduce from $3(m-2)^2$ parameters to $6$.
 We then use the $(\rho\s\t)$ equations to first eliminate the highest weight vector and then
 to eliminate the trivial module.
 
 Step 10: We show $T^{\rho mm}$, $T^{m\s m}$, and $T^{mm\t}$ are zero using the $(\rho\s m)$ equations.
 
 Step 11: We show the last unassigned $T^{ijk}$, namely $T^{mmm}$ is zero via the
 $(\rho mm)$, $(m\s m)$ and $(mm\t)$ equations to complete the proof.

\subsection{Preliminaries}
  
 We write out some of the equations:

 \begin{equation*}
  \begin{array}{ll}
(111) & u_m^1 + v_m^1 + w_m^1 = 0 \\
~\\
(11\tau) & u_\tau^1 + v_\tau^1 + w_m^\tau +u^1_mT^{m1\t}= 0 \\
(1\sigma1) & \ep^{\s} u_{\bar\s }^1  + v_m^\sigma + w_\sigma^1+u^1_mT^{m\s 1} = 0 \\
(\rho11) & u_m^\rho + \ep^\rho v_{\bar\rho}^1    + w_\rho^1 +v^1_mT^{\rho m 1}= 0 \\
~\\
(1\sigma \tau) & u_1^1 \delta^{\sigma\tau}+u^1_\rho T^{\rho\s\t} +u^1_m T^{m\s\t} + v_\tau^\sigma 
 + w_\sigma^\tau = 0 \\
(\rho 1 \tau) & u_\tau^\rho+ u^\rho_mT^{m1\t} + v_1^1 \delta^{\rho\tau} + v^1_\s T^{\rho\s\t}+ v^1_m T^{\rho m\t}   + w_\rho^\tau 
= 0 \\
(\rho\sigma1) & -\ep^{\s} u_{\ol\s}^\rho   + \ep^{\rho}v_{\ol \rho}^\sigma   + w_1^1\ep^\rho \d^{\s\ol{\rho}}
+w^1_\tau T^{\rho\s\t}+w^1_m T^{\rho\s m}+u^\rho_mT^{m\s 1} + v^\s_m T^{\rho m1} = 0\\
~\\
(\rho\sigma\tau) & u_{1}^{\rho} \delta^{\sigma\tau} + v_1^{\sigma} \delta^{\rho\tau} + w_1^\tau \ep^\rho\d^{\ol{\s}\rho}
+u^\rho_{\rho'}T^{\rho'\s\t}+u^\rho_{m}T^{m\s\t}+v^{\s}_{\s'}T^{\rho\s'\t}+v^{\s}_{m}T^{\rho m\t} + w^\t_{\t'}T^{\rho\s\t'}
+ w^\t_{m}T^{\rho\s m} =0 \\
~\\
(11m) & u_1^1 + v_1^1 + w_m^m = 0 \\
(1m1) & u_1^1 + v_m^m + w_1^1 +u^1_\rho T^{\rho m 1}= 0 \\
(m11) & u_m^m + v_1^1 + w_1^1+v^1_\s T^{m\s 1} + w^1_\t T^{m1\t} = 0 \\
~\\
(1mm) & u^1_\rho T^{\rho mm}+u^1_m T^{mmm}+ v_1^m + w_1^m = 0 \\
(m1m) & u_1^m +v^1_\s T^{m\s m} + v^1_m T^{mmm}+  w_1^m+ w^1_\t T^{m1\t}  = 0 \\
(mm1) & u_1^m + v_1^m+ w^1_\t T^{mm\t} + w^1_m T^{mmm}+ u_\rho^mT^{\rho m 1} +v^m_\s T^{m\s 1}  = 0 \\
~\\
(1\sigma m) & u^1_\rho T^{\rho\s m} + u^1_m T^{m\s m}+  v_1^\sigma + w_\sigma^m
  = 0 \\
(\rho 1 m) & u_1^\rho + v^1_\s T^{\rho\s m} + v^1_m T^{\rho mm} + w_\rho^m = 0 \\
(m\sigma1) & -\ep^\s u_{\bar\s }^m  + v_1^\sigma + w^1_\t T^{m\s\t} + w^1_m T^{m\s m} 
+u^m_mT^{m\s 1} + v^\s_{\s'}T^{m\s' 1}+ w^1_1 T^{m\s 1}
 = 0 \\
(\rho m1) & u_1^\rho + \ep^\rho v_{\bar\rho}^m  + w^1_\t T^{\rho m\t} + w^1_m T^{\rho mm}
+u^\rho_{\rho'}T^{\rho' m1}+ v^m_mT^{\rho m1}+ w^1_1T^{\rho m1} = 0 \\
(1m\tau) & u^1_\rho T^{\rho m \t} + u^1_m T^{m m\t}+ v_\tau^m + w_1^\tau = 0 \\
(m1\tau) & u_\tau^m +v^1_\s T^{m\s\t} + v^1_m T^{mm\t}+  w_1^\tau
+u^m_mT^{m1\t}+ v^1_1 T^{m1\t} + w^\t_{\t'}T^{m1\t'} = 0 \\
  \end{array}
\end{equation*}
 
 The asymmetry in some of these  expressions  is caused by our normalizations.

\subsection{Step 1}     
We  solve:
    \begin{align*}
(111)\ \  w^1_m&=-(u^1_m+v^1_m)\\
(11\t)\ \  w^\t_m&=-(u^1_\t+v^1_\t+u^1_mT^{m1\t})\\  
(\rho 11)\ \  u^\s_m&=-(w^1_\s+\ep^\s v^1_{\ol \s}+u^1_mT^{m\s 1})\\
(1\s 1)\ \  v^\s_m&=-(\ep^\s u^1_{\ol \s} + w^1_\s  +v^1_mT^{\s m1} )\\
(\rho 1\t)\ \ u^\rho_\t  &=-[w^\t_\rho + v^1_1\d^{\rho\t}+v^1_\s T^{\rho\s\t} + v^1_m T^{\rho m \t}-(w^1_\rho+\ep^\rho v^1_{\ol \rho}+u^1_mT^{m\rho 1})T^{m1\t}]\\
(1\s\t)\ \  v^\s_\t&=-[w^\t_\s+u^1_1\d^{\s\t}+u^1_\rho T^{\rho\s\t}+u^1_mT^{m\s\t}\\
(11m)\ \ w^m_m&=-(u^1_1+v^1_1)\\
(1m1)\ \ v^m_m&=-(u^1_1+w^1_1+u^1_\rho T^{\rho m1})\\
(m11)\ \ u^m_m&=- (v_1^1 + w_1^1+v^1_\s T^{m\s 1} + w^1_\t T^{m1\t}).
\end{align*}

 \subsection{Step 2} 
 Write $w^\rho_\s=x^\rho_\s+y^\rho_\s$ where $(x^\rho_\s)\in \fsp(M)\subset \fgl(M)$ and $(y^\rho_\s)\in \fsp(M)^c$
where $\fsp(M)^c\subset \fgl(M)$ is the complementary $\fsp(M)$-module to $\fsp(M)$ in $\fgl(M)$.
We have   $\binom{m-2}2$ relations on the $x^\rho_\s$: $\ep^\s x^\s_\rho+\ep^\rho x^{\ol\rho}_{\ol\s}=0$, which
may also be written $\ep^\s x^{\ol\s}_\rho-\ep^\rho x^{\ol\rho}_{\s}=0$.
The $\binom{(m-2)+1}2$ relations on the $y^\rho_\s$ are 
$\ep^\s y^\s_\rho-\ep^\rho y^{\ol\rho}_{\ol\s}=0$.
Thus if we consider

\begin{align*}
 (\rho\sigma1)\ 0= &  \ep^{\s} [w^{\ol\s}_\rho + v^1_1\d^{\rho{\ol\s}}+v^1_{\s '}T^{\rho{\s'}{\ol\s}} + v^1_m T^{\rho m {\ol\s}}-(w^1_\rho+\ep^\rho v^1_{\ol \rho}+u^1_mT^{m\rho 1})T^{m1{\ol\s}}] \\
 &
   - \ep^{\rho} [w^{\ol\rho}_\s+u^1_1\d^{\s{\ol\rho}}+u^1_{\rho'} T^{\rho'\s{\ol\rho}}+u^1_mT^{m\s{\ol\rho}}
]  + w_1^1\ep^\rho \d^{\s\ol{\rho}}
+w^1_\tau T^{\rho\s\t}+(u^1_m+v^1_m) T^{\rho\s m} \\
&
-(w^1_{\rho}+\ep^{\rho} v^1_{\ol {\rho}}+u^1_mT^{m{\rho} 1})T^{m\s 1} 
-(\ep^\s u^1_{\ol \s} + w^1_\s  +v^1_mT^{\s m1} ) T^{\rho m1}  \\
\end{align*}
  the $x^\rho_\s$ are eliminated 
and the $\binom{m-2}2$ equations exactly allow us to solve for the $\binom{m-2}2$ independent $y^\rho_\s$
in terms of the elements of $\fu$. 
We obtain
\begin{align*}
& \ep^\s y^{\ol\s}_\rho-\ep^\rho y^{\ol\rho}_\s \\
&=
   \ep^{\s} [  v^1_1\d^{\rho{\ol\s}}+   v^1_{\s '}T^{\rho{\s'}{\ol\s}} + v^1_m T^{\rho m {\ol\s}}-(w^1_\rho+\ep^\rho v^1_{\ol \rho}+u^1_mT^{m\rho 1})T^{m1{\ol\s}}] \\
 &
   - \ep^{\rho} [u^1_1\d^{\s{\ol\rho}}+ u^1_{\rho'} T^{\rho'\s{\ol\rho}}+u^1_mT^{m\s{\ol\rho}}
]  \\
&+ w_1^1\ep^\rho \d^{\s\ol{\rho}}
+w^1_\tau T^{\rho\s\t}+(u^1_m+v^1_m) T^{\rho\s m}-(w^1_{\rho}+\ep^{\rho} v^1_{\ol {\rho}}+u^1_mT^{m{\rho} 1})T^{m\s 1} 
-(\ep^\s u^1_{\ol \s} + w^1_\s  +v^1_mT^{\s m1} ) T^{\rho m1}  \\
&=
\d^{\rho{\ol\s}}(\ep^{\s}v^1_1+\ep^\rho w^1_1-\ep^{\rho}u^1_1)
  - \ep^{\rho}   u^1_{\rho'} T^{\rho'\s{\ol\rho}}- \ep^\s u^1_{\ol \s}T^{\rho m1} 
+u^1_m(\ep^{\s}T^{m\rho 1} T^{m1{\ol\s}}-\ep^\rho  T^{m\s{\ol\rho}}+  T^{\rho\s m}+ T^{m{\rho} 1}T^{m\s 1})\\
&
+
 v^1_{\s '}\ep^{\s}    T^{\rho{\s'}{\ol\s}}-v^1_{\ol \rho}(\ep^\s\ep^\rho T^{m1{\ol\s}}+\ep^\rho T^{m\s 1} )
 +v^1_m(\ep^{\s} T^{\rho m {\ol\s}}+ T^{\rho\s m}-T^{\s m1}   T^{\rho m1} )\\
 &+w^1_\rho(\ep^{\s}T^{m1{\ol\s}}-T^{m\s 1})-w^1_\s T^{\rho m1}  
 +w^1_\tau T^{\rho\s\t} .
\end{align*}
   
\subsection{Step 3}
Consider $(\rho\s\t)$ 
\begin{align*}
  0=& u_{1}^{\rho} \delta^{\sigma\tau} + v_1^{\sigma} \delta^{\rho\tau} + w_1^\tau \ep^\rho\d^{\ol{\s}\rho}\\
&-[w^{\rho'}_\rho + v^1_1\d^{\rho{\rho'}}+v^1_{\s'} T^{\rho\s'{\rho'}} + v^1_m T^{\rho m {\rho'}}-(w^1_\rho+\ep^\rho v^1_{\ol \rho}+u^1_mT^{m\rho 1})T^{m1{\rho'}}]T^{\rho'\s\t}\\
& -(w^1_{\rho}+\ep^{\rho} v^1_{\ol {\rho}}+u^1_mT^{m{\rho} 1})T^{m\s\t}\\
&-[w^{\s'}_\s+u^1_1\d^{\s{\s'}}+u^1_{\rho'} T^{\rho'\s{\s'}}+u^1_mT^{m\s{\s'}}
-(w^1_\s+\ep^\s u^1_{\ol \s}+v^1_mT^{m\s 1})T^{1m{\s'}}]T^{\rho\s'\t}\\
&-(\ep^\s u^1_{\ol \s} + w^1_\s  +v^1_mT^{\s m1} )T^{\rho m\t}\\
& + w^\t_{\t'}T^{\rho\s\t'}
  -(u^1_\t+v^1_\t+u^1_mT^{m1\t})T^{\rho\s m}   \\
\end{align*}
Specialize to where  $\rho\neq \ol\s$, $\rho\neq \t$, $\s\neq \t$ and,
 after writing $w^\rho_\s=x^\rho_\s+y^\rho_\s$,  notice that modulo 
 $\fu$, we just have
the action of $\fsp(M)_C$ on $M\ot M\ot M$.
The stabilizers of a highest weight vector (and hence any vector) in the modules
$M_{\omega_3}$ and $M_{\omega_1+\omega_2}$ have codimension greater than $ m-1 $, so
$T^{\rho\s\t}a_\rho\ot b_\s \ot c_\t$ cannot have nonzero components in these modules.
Moreover, for the $S^3M$ component, unless it  is a highest weight vector, the case is similarly eliminated.
The three copies corresponding to $M$ are exactly the three cases $\rho=\ol\s$, $\rho=\t$, $\s=\t$
and our analysis says nothing about them so far.

Thus the only components that are possibly nonzero in $M\ot M\ot M$ are    $S^3M$ and the three copies of  $M$,
and we may assume the component in $S^3M$ is a highest weight vector which we take to be
$a_2\ot b_2\ot c_{p+2}$.
That is
\begin{align*}
&T^{\rho\s\t}a_\rho \ot b_\s \ot c_\t\\
&= 
T^{22 p+2}a_2\ot b_2\ot c_{p+2} + \ep^\rho t_{C}^{\rho\ol\rho\t'}a_\rho\ot b_{\ol \rho} \ot c_{\t'}
+t_A^{\rho'\s\s}  a_{\rho'}\ot b_\s\ot c_\s + t_B^{\rho \s'\rho }   a_\rho \ot b_{\s'}\ot c_\rho\\
&:=T^{22 p+2}a_2\ot b_2\ot c_{p+2} +  t_{C}^{ \t'}\sum_\rho \ep^\rho a_\rho\ot b_{\ol \rho} \ot c_{\t'}
+t_A^{\rho'} \sum_\s a_{\rho'}\ot b_\s\ot c_\s + t_B^{  \s'  } \sum_\rho  a_\rho \ot b_{\s'}\ot c_\rho
\end{align*}where $t_{C}^{\rho\ol\rho\t'}$ is independent of $\rho$,
$t_A^{\rho'\s\s} $ is   independent of $\s$, $t_B^{\rho \s'\rho }$
 is independent of $\rho$.
 
The parabolic stabilizing the highest weight vector in $S^3M$ has codimension
$m-3$ in $\fsp(M)$,
  so if there is a drop of dimension by two more, we must have $T^{22p+2}=0$.

\subsection{Step 4} Consider $(\rho\s 1)+(\s\rho 1)$ to eliminate the $w^\t_{\t'}$:
\begin{align*}
0&=  \ep^{\s} [  v^1_1\d^{\rho{\ol\s}}+v^1_{\s '}T^{\rho{\s'}{\ol\s}} + v^1_m T^{\rho m {\ol\s}}-(w^1_\rho+\ep^\rho v^1_{\ol \rho}+u^1_mT^{m\rho 1})T^{m1{\ol\s}}] \\
 &
   + \ep^{\rho} [- u^1_1\d^{\s{\ol\rho}}-u^1_{\rho'} T^{\rho'\s{\ol\rho}}-u^1_mT^{m\s{\ol\rho}}
]  \\
&
 + w_1^1\ep^\rho \d^{\s\ol{\rho}}
+w^1_\tau T^{\rho\s\t}+w^1_m T^{\rho\s m}-(w^1_{\rho}+\ep^{\rho} v^1_{\ol {\rho}}+u^1_mT^{m{\rho} 1})T^{m\s 1} 
-(\ep^\s u^1_{\ol \s} + w^1_\s  +v^1_mT^{\s m1} ) T^{\rho m1}  \\
& + \ep^{\rho} [    v^1_1\d^{\s{\ol\rho}}+v^1_{\rho '}T^{\s{\rho'}{\ol\rho}} + v^1_m T^{\s m {\ol\rho}}-(w^1_\s+\ep^\s v^1_{\ol \s}+u^1_mT^{m\s 1})T^{m1{\ol\rho}}] \\
 &
   + \ep^{\s} [ -u^1_1\d^{\rho{\ol\s}}-u^1_{\s'} T^{\s'\rho{\ol\s}}-u^1_mT^{m\rho{\ol\s}}
]  \\
&
 + w_1^1\ep^\s \d^{\rho\ol{\s}}
+w^1_\tau T^{\s\rho\t}+w^1_m T^{\s\rho m}-(w^1_{\s}+\ep^{\s} v^1_{\ol {\s}}+u^1_mT^{m{\s} 1})T^{m\rho 1} 
-(\ep^\rho u^1_{\ol \rho} + w^1_\rho  +v^1_mT^{\rho m1} ) T^{\s m1}  \\
&= \\
&
-  u^1_1 \d^{\rho{\ol\s}}(\ep^{\s} + \ep^{\rho}  )
+ u^1_{\ol \rho}\ep^\rho (  -  T^{\s m1} -  t^\s_B+\ep^\s  t^{\ol \s}_C )
 +  u^1_{\ol \s} \ep^\s ( -  T^{\rho m 1}+   \ep^\rho t^{\ol\rho}_C-  t^{\rho}_B)\\
 &-2u^1_2\d_{\s 2}\d_{\rho 2}T^{22(p+2)}- u^1_{\rho'} (\ep^\rho+\ep^\s) \d_{\s\ol\rho}t^{\rho'}_A\\
 &
+u^1_m(-\ep^\s T^{m\rho 1}T^{m1\ol\s}-2T^{m\rho 1}T^{m\s 1}-\ep^{\rho}T^{m\s{\ol\rho}}-\ep^{\s} T^{m\rho{\ol\s}}
-\ep^\rho T^{m\s 1}T^{m1\rho}-T^{\rho\s m}-  T^{\s\rho m})\\
&
+ v^1_1 \d^{\rho{\ol\s}}(\ep^{\s} + \ep^{\rho}  )
+v^1_{\ol {\rho}} \ep^{\rho} (-\ep^\s  T^{m1\ol \s}-  T^{m\s 1}+ \ep^\s t^{\ol\s}_C+  t^{\s}_A)
 + v^1_{\ol {\s}}\ep^{\s}(- T^{m\rho 1} - \ep^\rho T^{m1\ol\rho}+  t^\rho_A+ \ep^\rho  t^{\ol\rho}_C)\\
 &
 +2v^1_2\d_{\rho 2}\d_{\s 2} T^{22(p+2)}+v^1_{\s'}(\ep^\s+\ep^\rho)\d_{\rho\ol\s} t^{\s'}_B\\
 &
+v^1_m( \ep^\rho T^{m\s 1}T^{1m\ol\rho}+  \ep^{\s}T^{\rho m {\ol\s}}+
2T^{\s m 1}T^{\rho m1}+\ep^\rho T^{\s m {\ol\rho}}+\ep^\s T^{m\rho 1}T^{1m\ol\s}-T^{\rho\s m}-  T^{\s\rho m})
\\
&
  + w_1^1\d^{\s\ol{\rho}}(\ep^\rho +  \ep^\s   )
+  w^1_\rho(- \ep^{\s}T^{m1{\ol\s}}     - T^{m\s 1}     -   T^{\s m1}+t^\s_B+t^\s_A )
+ w^1_\s (  - T^{\rho m1}- \ep^{\rho}T^{m1{\ol\rho}} -  T^{m\rho 1}+t^\rho_A+t^\rho_B ) \\
&
+w^1_\tau \d_{\rho\ol\s}(\ep^\rho+\ep^\s)t^\t_C+ 2w^1_{p+2}\d_{p2}\d_{\s 2}T^{22(p+2)}
\end{align*}
Note that when $\rho=\ol \s$, $\ep^\rho+\ep^\s=0$, so all the $\d_{\rho\ol\s}(\ep^\s+\ep^\rho)$
terms are zero.

When $(\rho\s)=(2,2)$, if $T^{22(p+2)}\neq 0$ we obtain a relation involving $u^1_2,v^1_2,w^1_{p+2}$
so there can be no further relations when $T^{22(p+2)}\neq 0$.
Moreover, when  $T^{22(p+2)}\neq 0$, we must also have   for all $ \s$
\begin{align*}
0&= -T^{\s m1} -  t^\s_B+\ep^\s t^{\ol \s}_C\\
0&=- \ep^{\s}T^{m1{\ol\s}}     - T^{m\s 1}     -   T^{\s m1}+t^\s_B+t^\s_A\\
0&=-\ep^\s  T^{m1\ol \s}-  T^{m\s 1}+ \ep^\s t^{\ol\s}_C+  t^{\s}_A
\end{align*}
which we rewrite as
\begin{align*}
T^{ \s m1}&=0\\
t^{ \s}_B&= \ep^\s t^{\ol\s}_C\\
 t^\s_A&=\ep^{\s}T^{m1{\ol\s}} +T^{m\s 1}   -  \ep^\s t^{\ol\s}_C
 \end{align*}
 This is because the terms appear inside coefficients that involve $\rho$ only, so even though
 they also appear in the relation involving $u^1_2,v^1_2,w^1_{p+2}$, we can peel that away
 from the others to get the relation   
 $-u^1_2+v^1_2+w^1_{p+2}=0
 $.

\subsection{Step 5}
  Assume $T^{22(p+2)}\neq 0$
and  reconsider $(\rho\s\t)=(22(p+2))$ modulo $\fu$.
Recall that when $T^{22(p+2)}\neq 0$, relations are imposed upon
$x^2_3\hd x^{m-2}_2$ but $x^2_2\in \fsp(M)$ is still free.

\begin{align*}
  0\equiv  &
  -[w^{(p+2)}_2    ] (-1) t^{p+2}_C 
    -[w^{2}_2 + v^1_1 ] T^{22(p+2)}\\
 &-[w^{(p+2)}_2   ] t^{p+2}_C 
 -[w^{(p+2)}_2 ] t^{p+2}_C\\
&-[w^{2}_2+u^1_1  ]  T^{22(p+2)}  -[w^{(p+2)}_2 ] t^2_A\\
& + w^{(p+2)}_{2}t^2_B +w^{(p+2)}_{p+2}T^{22(p+2)} + w^{(p+2)}_{2}t^2_A  
\ \   \tmod \fn
  \\
  & \equiv  
   ( -2w^{2}_2  +w^{(p+2)}_{p+2})T^{22(p+2)}  \ \   \tmod \ \fu\\
 & \equiv  
   ( 3x^2_2+y^2_2)T^{22(p+2)}    \ \   \tmod  \fu\\
 & \equiv  
   ( 3x^2_2 )T^{22(p+2)}   \ \   \tmod  \fu
\end{align*} 
Recall that so far there was no relation on $x^2_2$ so we obtain a new relation and thus $T^{22(p+2)}=0$.

\subsection{Step 6}
We revisit $(\rho\s 1)+(\s\rho 1)$ which simplifies since $T^{22(p+2)}=0$:
\begin{align*}
0&=  
+ u^1_{\ol \rho}\ep^\rho (  -  T^{\s m1} -  t^\s_B+\ep^\s  t^{\ol \s}_C )
 +  u^1_{\ol \s} \ep^\s ( -  T^{\rho m 1}+   \ep^\rho t^{\ol\rho}_C-  t^{\rho}_B)\\
 &
+u^1_m(-\ep^\s T^{m\rho 1}T^{m1\ol\s}-2T^{m\rho 1}T^{m\s 1}-\ep^{\rho}T^{m\s{\ol\rho}}-\ep^{\s} T^{m\rho{\ol\s}}
-\ep^\rho T^{m\s 1}T^{m1\rho}-T^{\rho\s m}-  T^{\s\rho m})\\
&
+v^1_{\ol {\rho}} \ep^{\rho} (-\ep^\s  T^{m1\ol \s}-  T^{m\s 1}+ \ep^\s t^{\ol\s}_C+  t^{\s}_A)
 + v^1_{\ol {\s}}\ep^{\s}(- T^{m\rho 1} - \ep^\rho T^{m1\ol\rho}+  t^\rho_A+ \ep^\rho  t^{\ol\rho}_C)\\
 &
+v^1_m( \ep^\rho T^{m\s 1}T^{1m\ol\rho}+  \ep^{\s}T^{\rho m {\ol\s}}+
2T^{\s m 1}T^{\rho m1}+\ep^\rho T^{\s m {\ol\rho}}+\ep^\s T^{m\rho 1}T^{1m\ol\s}-T^{\rho\s m}-  T^{\s\rho m})
\\
&
+  w^1_\rho(- \ep^{\s}T^{m1{\ol\s}}     - T^{m\s 1}     -   T^{\s m1}+t^\s_B+t^\s_A )
+ w^1_\s (  - T^{\rho m1}- \ep^{\rho}T^{m1{\ol\rho}} -  T^{m\rho 1}+t^\rho_A+t^\rho_B )
\end{align*}

If any of the quantities
\begin{align}
\nonumber & -  T^{\s m1} -  t^\s_B+\ep^\s  t^{\ol \s}_C\\
\label{stara} & -\ep^\s  T^{m1\ol \s}-  T^{m\s 1}+ \ep^\s t^{\ol\s}_C+  t^{\s}_A\\
\nonumber &- \ep^{\s}T^{m1{\ol\s}}     - T^{m\s 1}     -   T^{\s m1}+t^\s_B+t^\s_A 
\end{align}
is not identically zero, we obtain $m-2$ relations and can have no more. In particular
at most one of these quantities is nonzero and if it is, there can be no further relations.

\bigskip

\subsection{Step 7}
Write the $6$ equations $(1\s m)\hd (m1\t)$
as
\begin{align*}
&X_{1\s m}+v^\s_1+w^m_\s=0\\
&X_{\s 1 m}+ u^\s_1+w^m_\s=0\\
&X_{m \s 1  }- \ep^\s u^m_{\ol \s}+v^\s_1=0\\
&X_{  \s  m1  }+u^\s_1+ \ep^\s v^m_{\ol \s} =0\\
&X_{1m\ol \s }+ v^m_{\ol \s}+w^{\ol\s}_1=0\\
&X_{m1\ol \s }+ u^m_{\ol \s}+w^{\ol\s}_1=0
\end{align*}
Then  
\begin{align*}
u^\s_1&= \frac 12[ -X_{1\s m}+X_{ \s 1m}-\ep^\s X_{m\s 1}+\ep^\s X_{ \s m1}+\ep^\s X_{1m\ol \s } +X_{m1\ol \s} ]\\
v^\s_1&= \frac 12[  X_{1\s m}-X_{ \s 1m}+\ep^\s X_{m\s 1}-\ep^\s X_{ \s m1}-\ep^\s X_{1m\ol \s } +X_{m1\ol \s} ] \\
w^{\ol \s}_1&=  \frac 12[  X_{1\s m}+X_{ \s 1m}+\ep^\s X_{m\s 1}-\ep^\s X_{ \s m1}-\ep^\s X_{1m\ol \s } -X_{m1\ol \s} ]\\
u^m_{\ol \s}&=  \frac 12[  X_{1\s m}+X_{ \s 1m}-\ep^\s X_{m\s 1}+\ep^\s X_{ \s m1}+\ep^\s X_{1m\ol \s } -X_{m1\ol \s} ]\\
v^m_{\ol \s}&= \frac 12[ -\ep^\s X_{1\s m}-\ep^\s X_{ \s 1m}+ X_{m\s 1}-  X_{ \s m1}+  X_{1m\ol \s } +\ep^\s X_{m1\ol \s} ] \\
w^m_\s&= \frac 12[ \ep^\s X_{1\s m}+\ep^\s X_{ \s 1m}+ X_{m\s 1}+  X_{ \s m1}-  X_{1m\ol \s } -\ep^\s X_{m1\ol \s} ]
\end{align*}

Observe that
\begin{align*}
&X_{1\s m}\equiv 0 \tmod\fn \\
&X_{\s 1 m}\equiv 0 \tmod\fn \\
&X_{m \s 1  }\equiv -x^{\s'}_\s T^{m\s' 1}+[-(u^1_1+v^1_1+2w^1_1)+\frac 12 \ep^{\s}(u^1_1+v^1_1-w^1_1)]T^{m\s 1} \tmod\fn \\
&X_{  \s  m1  }\equiv -x^{\s'}_\s T^{ \s' m1}
+[-(u^1_1+v^1_1+2w^1_1)+\frac 12 \ep^{\s}(u^1_1+v^1_1-w^1_1)]T^{\s m  1}  \tmod\fn \\
&X_{1m\ol \s }\equiv 0 \tmod\fu \\
&X_{m1\ol \s }\equiv -\ep^{\s}\ep^{\s'} x^{\s'}_\s T^{m1\s' } 
 +[- 2w^1_1 +\frac 12 \ep^{\s}(u^1_1+v^1_1-w^1_1)]T^{m1\s } \tmod\fn 
\end{align*}

\subsection{Step 8}
Now consider $(\rho\s\t)\tmod \fu$:   
\begin{align*}
0\equiv &\frac 12[ \ep^\rho  x^{\s'}_\rho T^{m\s' 1} 
-\ep^\rho x^{\s'}_\rho T^{ \s' m1}    -\ep^{\rho}\ep^{\s'} x^{\s'}_\rho T^{m1\s' } ]\d_{\s\t}\\
+
&\frac 12[  - \ep^\s  x^{\s'}_\s T^{m\s' 1}+\ep^\s  x^{\s'}_\s T^{ \s' m1}  -\ep^{\s}\ep^{\s'} x^{\s'}_\s T^{m1\s' }  ]\d_{\rho\t}\\
+
& \frac 12[  -\ep^{\t}  x^{\s'}_{\t}  T^{m\s' 1} +\ep^{\t}   x^{\s'}_\t T^{ \s' m1}  +\ep^{\t}\ep^{\s'} x^{\s'}_\t T^{m1\s' }  ]\d_{\ol\s\rho}\\
&-x^{\s'}_\rho \d_{\s\t}t^{\s'}_A
-x^{\s'}_\s \d_{\rho\t} t^{\s'}_B
-x_\t^{\s'}\d_{\ol\s \rho}\ep^{\rho}t^{\s'}_C\\
&+  [-\ep^\rho[-(u^1_1+v^1_1+2w^1_1)+\frac 12 \ep^{\rho}(u^1_1+v^1_1-w^1_1)]T^{m\rho 1}
+\ep^\rho[-(u^1_1+v^1_1+2w^1_1)+\frac 12 \ep^{\rho}(u^1_1+v^1_1-w^1_1)]T^{\rho m  1} \\
&
+[- 2w^1_1 +\frac 12 \ep^{\rho}(u^1_1+v^1_1-w^1_1)]T^{m1\rho }]\d_{\s\t}\\
&
+[\ep^\s[-(u^1_1+v^1_1+2w^1_1)+\frac 12 \ep^{\s}(u^1_1+v^1_1-w^1_1)]T^{m\s 1}
-\ep^\s[-(u^1_1+v^1_1+2w^1_1)+\frac 12 \ep^{\s}(u^1_1+v^1_1-w^1_1)]T^{\s m  1} \\
&
+[- 2w^1_1 +\frac 12 \ep^{\s}(u^1_1+v^1_1-w^1_1)]T^{m1\s }]\d_{\rho\t}\\
&+  [ \ep^\t[-(u^1_1+v^1_1+2w^1_1)+\frac 12 \ep^{\t}(u^1_1+v^1_1-w^1_1)]T^{m\t 1}
-\ep^\t[-(u^1_1+v^1_1+2w^1_1)+\frac 12 \ep^{\t}(u^1_1+v^1_1-w^1_1)]T^{\t m  1} \\
&
-[- 2w^1_1 +\frac 12 \ep^{\t}(u^1_1+v^1_1-w^1_1)]T^{m1\t }]\d_{\ol\s\rho}\\
&=\\
&
x^{\s'}_\rho\d_{\s\t}[-t^{\s'}_A+\frac 12( \ep^\rho    T^{m\s' 1}-
\ep^\rho    T^{ \s' m1}    -\ep^{\rho}\ep^{\s'}   T^{m1\s' } )]\\
&+
x^{\s'}_\s\d_{\rho\t}[ -t^{\s'}_B+ \frac 12( - \ep^\s   T^{m\s' 1}+\ep^\s     T^{ \s' m1}  -\ep^{\s}\ep^{\s'}  T^{m1\s' })  ]\\
&
+x^{\s'}_{\t} \d_{\ol\s\rho}[ -\ep^{\rho}t^{\s'}_C+\frac 12(  -\ep^{\t}   T^{m\s' 1} +\ep^{\t}    T^{ \s' m1}  +\ep^{\t}\ep^{\s'}  T^{m1\s' })  ]
\\
&
(u^1_1+v^1_1)[(\ep^\rho +\frac 12)  T^{m\rho 1}
+(-\ep^\rho +\frac 12 )T^{\rho m  1} 
 +\frac 12 \ep^{\rho} T^{m1\rho }]\d_{\s\t}\\
 &+
w^1_1[(-2\ep^\rho -\frac 12 )T^{m\rho 1}+(-2\ep^\rho-\frac 12)T^{\rho m  1} 
+(-2-\frac 12 \ep^{\rho}T^{m1\rho })
]\d_{\s\t}\\
&
+(u^1_1+v^1_1)[(-\ep^\s +\frac 12 \ep^{\s})T^{m\s 1}
+(\ep^\s-\frac 12 \ep^{\s})T^{\s m  1}   +\frac 12 \ep^{\s} T^{m1\s }]\d_{\rho\t}\\
&
+w^1_1[(2\ep^\s -\frac 12 \ep^{\s})T^{m\s 1}
+\ep^\s(2+\frac 12 \ep^{\s})T^{\s m  1}  +( - 2 -\frac 12 \ep^{\s})T^{m1\s }]\d_{\rho\t}\\
&+ (u^1_1+v^1_1) [ (-\ep^\t +\frac 12 \ep^{\t})T^{m\t 1}
-\ep^\t(-1+\frac 12 \ep^{\t})T^{\t m  1} - \frac 12 \ep^{\t}T^{m1\t }]\d_{\ol\s\rho}\\
&+ w^1_1 [(2-\frac 12 \ep^{\t})T^{m\t 1}
+(-2\ep^\t -\frac 12 \ep^{\t})T^{\t m  1}+(  2 +\frac 12 \ep^{\t})T^{m1\t }]\d_{\ol\s\rho} 
\end{align*}
 We already know that at least two of the quantities from \eqref{stara} must be zero.
 If exactly two are zero then all the terms in brackets above must be zero and
 we conclude all six quantities are zero (for all $\s$). If all three of the quantities are
 zero then we are allowed $m-2$ relations among the elements of $\fsp(M)\op \ft$.
 Note each term in brackets with the $x$'s  is two possible relations, depending on whether,
 e.g., in the first $\rho<p+1$ or $\rho\geq p+1$ and each non-vanishing imposes $\frac{m-2}2$ relations.
 Thus two of the six can fail to be identities. Then the remaining four identities plus the previous
 three will be enough to set all quantities to zero. 

 We conclude
$t^\rho_A,t^\s_B,t^\t_C, T^{m1\s},T^{m\s 1},  T^{\s m 1}$ are all zero.
 
\subsection{Step 9}
We now show $T^{\rho\s m}$ and permutations are zero.
We may now solve the $(1mm),(m1m),(mm1)$ equations to
obtain  $u^m_1, v^m_1,w^m_1\equiv 0\tmod \fn$.
Using this we obtain from the $(\rho\s m)$ equations
$$
-x^{\rho'}_\rho T^{\rho'\s m}-x^{\s'}_\s T^{\rho \s' m}\equiv 0 \tmod \fu.
$$
Writing $M\ot M=S^2M\op \langle \Lambda\rangle \op  \La 2 M_0$,
where $\La 2 M_0\op  \langle \Lambda\rangle=\La 2 M$, arguing as before, we immediately
see that the component of $T^{\rho\s m}a_\rho\ot b_\s$ lying in $\La 2 M_0$
must be zero and the component in $S^2M$ must be a highest weight vector, i.e.,
$$
T^{\rho\s m}a_\rho\ot b_\s=T^{22m}a_2\ot b_2\ot c_m\op \ep^\rho t^{** m}\sum_\rho a_\rho\ot b_{\ol\rho}\ot c_m
$$
and if $T^{22m}\neq 0$, there can be at most two more relations imposed.
Similarly, using the $(\rho m\t)$ and $(m\s\t)$ equations, and writing, as a $\fsp(M)$-module, 
$M^*\ot M=M\ot M=S^2M\op \langle \Lambda\rangle \op  \La 2 M_0$, we obtain
\begin{align*}
T^{\rho m\t}a_\rho\ot c_\t &=T^{2m(p+2)}a_2\ot b_m\ot c_{p+2}\op   t^{* m*  }\sum_{\rho}a_\rho\ot b_{m}\ot c_\rho\\
T^{m\s\t}b_\s\ot c_\t&=T^{m2(p+2)}a_m\ot b_2\ot c_{p+2}\op   t^{m**}\sum_\s a_m\ot b_{\s}\ot c_\s
\end{align*}

If any one of $T^{22m},T^{2m(p+2)},T^{m2(p+2)}$
  is nonzero we immediately get $m-3$ relations and just need two more.

Reonsider $(\rho\s\t)$ 
\begin{align*}
  0=& u_{1}^{\rho} \delta^{\sigma\tau} + v_1^{\sigma} \delta^{\rho\tau} + w_1^\tau \ep^\rho\d^{\ol{\s}\rho}\\
& -(w^1_{\rho}+\ep^{\rho} v^1_{\ol {\rho}} )\d_{\s\t}t^{m**}-(w^1_{\rho}+\ep^{\rho} v^1_{\ol {\rho}} )\d_{\s 2}\d_{\t p+2}T^{m2(p+2)}\\
& 
 -(\ep^\s u^1_{\ol \s} + w^1_\s  )\d_{\rho\t} t^{* m*}  -(\ep^\s u^1_{\ol \s} + w^1_\s  )\d_{\rho 2}\d_{\t(p+2)}T^{2 m(p+2)} \\
&  -(u^1_\t+v^1_\t )\d_{\s\ol\rho} \ep^\rho t^{** m} -(u^1_\t+v^1_\t )\d_{\rho 2}\d_{\s 2}T^{22 m} 
\end{align*}
Respectively taking $(\rho\s\t)$ to be $(\rho 2(p+2))$ with $\rho\neq p+2$, $(2\s (p+2))$
with $\s\neq p+2$,  and $(22\t)$ with $\t\neq 2$, we see
$T^{m2(p+2)}=0$, $T^{2m(p+2)}$, and $T^{22m}=0$.

At this point 
\begin{align*}
u^\rho_1&= \frac 12[  \ep^\rho u^1_\rho t^{m**}+\ep^\rho v^1_\rho t^{**m}-\ep^\rho w^1_\rho t^{m**}+\ep^\rho w^1_\rho t^{*m*}+\ep^\rho u^1_{\ol\rho}t^{*m*}+v^1_{\ol\rho}t^{m**} ]\\
v^\s_1&= \frac 12[  -\ep^\s u^1_\s t^{m**}-\ep^\s v^1_\s t^{**m}+\ep^\s w^1_\s t^{m**}-\ep^\s w^1_\s t^{*m*}-\ep^\s u^1_{\ol\s}t^{*m*}+v^1_{\ol\s}t^{m**} ] \\
w^{\ol \t}_1&=  \frac 12[  -\ep^\t u^1_\t t^{m**}+\ep^\t v^1_\t t^{**m}+\ep^\t w^1_\t t^{m**}-\ep^\t w^1_\t t^{*m*}-\ep^\t u^1_{\ol\t}t^{*m*}-v^1_{\ol\t}t^{m**} ]
\end{align*}

Subbing in to $(\rho\s\ol\t)$:

\begin{align*}
  0=& \frac 12[  \ep^\rho u^1_\rho t^{m**}+\ep^\rho v^1_\rho t^{**m}-\ep^\rho w^1_\rho t^{m**}+\ep^\rho w^1_\rho t^{*m*}+\ep^\rho u^1_{\ol\rho}t^{*m*}+v^1_{\ol\rho}t^{m**} ] \delta^{\sigma\ol\tau} \\
&  +\frac 12[  -\ep^\s u^1_\s t^{m**}-\ep^\s v^1_\s t^{**m}+\ep^\s w^1_\s t^{m**}-\ep^\s w^1_\s t^{*m*}-\ep^\s u^1_{\ol\s}t^{*m*}+v^1_{\ol\s}t^{m**} ] \delta^{\rho\ol\tau}\\
  & +  \frac 12[  -\ep^\t u^1_\t t^{m**}+\ep^\t v^1_\t t^{**m}+\ep^\t w^1_\t t^{m**}-\ep^\t w^1_\t t^{*m*}-\ep^\t u^1_{\ol\t}t^{*m*}-v^1_{\ol\t}t^{m**} ] \ep^\rho\d^{\ol{\s}\rho}\\
& -(w^1_{\rho}+\ep^{\rho} v^1_{\ol {\rho}} )\d_{\s\ol\t}t^{m**}   -(\ep^\s u^1_{\ol \s} + w^1_\s  )\d_{\rho\ol\t} t^{* m*}   -(u^1_{\ol\t}+v^1_{\ol\t })\d_{\s\ol\rho} \ep^\rho t^{** m}   
\end{align*}

Recall that if any of $t^{m**},t^{*m*},t^{**m}$ are nonzero, we have $m-3$ relations imposed on 
the elements of $\fsp(M)$ so we only need to show these equations imply at least two more.
Say $t^{m**}\neq 0$, then for $\rho\neq \s,\ol\s$ and taking $\t=\ol\s$, we get at least $\frac 12(m-4)$ relations
among the $u^1_\rho$ (in fact more unless $t^{m**}=t^{*m*}$)  and are done because $\frac 32 m-5\geq m-1$ as $m\geq 7$ and $m$ is even. The other two cases are similar. 
 
\subsection{Step 10}
Reconsider
\begin{align*}
(\rho\s m)\ \ 0&= u^\rho_mT^{m\s m}+v^\s_mT^{\rho mm}\\
&= -(w^1_\rho +\ep^\rho v^1_{\ol\rho})T^{m\s m} -(\ep^\s u^1_{\ol\s} + w^1_\s)T^{\rho mm}\\
(\rho m\tau )\ \ 0&= u^\rho_mT^{mm\t}+w^\t_mT^{\rho mm}\\
&= -(w^1_\rho +\ep^\rho v^1_{\ol\rho})T^{mm\t} -(  u^1_\t + v^1_\t)T^{\rho mm}\\
(m\s\t)\ \ 0&=  v^\s_mT^{mm\t}+ w^\t_m T^{m\s m}\\
&= -(w^1_\rho +\ep^\rho v^1_{\ol\rho})T^{m\s m} -(  u^1_\t + v^1_\t)T^{m\s m}
\end{align*}
Note that if $T^{\rho mm}\neq 0$ for some $\rho$, then the $u^1_{\ol\s}$ in the first
line and the $u^1_\t$ in the second give $2(m-2)$ relations, eliminating this case.
Similarly,  if any of $T^{\rho mm}$, $T^{m\s m}$, $T^{mm\t}$ are nonzero, the
case is eliminated from consideration.

\subsection{Step 11}
\begin{align*}
(\rho mm)\ \ 0&= u^\rho_m T^{mmm} \\
&=
-(w^1_\rho+\ep^\rho v^1_{\ol\rho}) T^{mmm}
\end{align*}
\begin{align*}
(m\s m)\ \ 0&= v^s_m T^{mmm}\\
 &=
  -(\ep^\s  u^1_{\ol\s} + w^1_\s)T^{mmm} \\
 & \\
 (mm\t)  \ \ 
 0&= -(u^1_\t+v^1_\t)T^{mmm}
  \end{align*}
So if $T^{mmm}\neq 0$ we are eliminated from consideration.

\qed

\section{Border rank bounds}\label{brbnds}
The rank of a tensor $T$, denoted $\bold R(T)$,  is the smallest $r$ such that $T$ may
be written as a sum of $r$ rank one tensors, and the border rank, denoted $\ur(T)$,
is the smallest $r$ such that $T$ is a limit of rank $r$ tensors.
Rank and border rank are standard measures of the complexity of a tensor.
  Strassen \cite{StrassenGaussEl} showed that the exponent $\omega$ of 
matrix multiplication may be defined as the infimum over $\tau$ such that 
$\bold 
R(\Mn)=O(\nnn^{\tau })$, and Bini \cite{MR605920} showed   one may  use  the 
border rank $\uR(\Mn)$ rather than the rank $\bfR(\Mn)$ in the definition. The 
tensor $T_{CW,m-2}$ has the minimal possible border rank $m$ for any concise 
tensor, which is important for its use in proving upper bounds on $\omega$.

\begin{remark}
The tensor of Proposition \ref{symthma}
satisfies $\bfR(a_1\ot (\sum_{j=1}^m b_j\ot c_j))=\uR(a_1\ot (\sum_{j=1}^m 
b_j\ot c_j))=m$.
\end{remark}

Let $T=S_\bbta$ be a skeletal tensor. Change bases in $B$ by permuting $b_1$ with $b_m$.
This will have the advantage that $T(\a^1)=\Id$.
Explicitly
$$
T(A^*)=
\begin{pmatrix}
\a^1&\a^2&\a^3&\cdots & & \a^m\\
&\a^1   & & & &\phi^2\\
& & \ddots & & &\phi^3\\
 & & &   & &\vdots\\
  & & &   & \ddots &\phi^{m-1}\\
  & & & & &\a^1
\end{pmatrix}, 
$$
where $\ol\phi =\ol\bbta(\a^2\hd \a^{m-1})^{\bt}$, and  $\ol\bbta$ is the matrix of $\bbta$.
Let $Y=\sum_{s=2}^{m-2}y_s\a^s$ and $Z=\sum_{s=2}^{m-2}z_s\a^s$
for constants $\ol y=(y_2\hd y_{m-1})$ and $\ol z=(z_2\hd z_{m-1})$.
Applying Strassen's commutation equations, the $(1,m)$ entry of $[T(Y),T(Z)]$
is the only potential nonzero entry and it is
$\bbta(\ol y,\ol z)-\bbta(\ol z,\ol y)$.
Note that if $\bbta$ is symmetric, then $T$ is isomorphic to $T_{CW}$. We conclude:

\begin{proposition}\label{brbndprop} Let
$T$ be skeletal. Then $\ur(T)\geq m+1$ unless $T$ is isomorphic to the big Coppersmith-Winograd tensor.
\end{proposition}

\begin{corollary} None of  $T_{skewCW,m-2}$, $T_{s\oplus skewCW,m-2}$, $T_{s+skewCW,m}$
have minimal border rank $m$.
\end{corollary}
  
\begin{corollary}\label{newcwchar}  Let $T\in \BC^m\ot \BC^m\ot \BC^m$ be 
$1$-generic and either symmetric or of minimal border rank. Then $\dim G_T\leq 
\binom{m+1}{2}$  with equality holding only for $T_{CW,m-2}$. 
\end{corollary}

Numerical computations using ALS methods (see \cite{MR3926208}) indicate, at 
least
for $m\leq 11$, that $\ur(T_{s\oplus skewCW,m-2})\leq \frac{3m}2-\frac 1 2$
and for $m\leq 14$ that  $\ur(T_{skewCW,m-2})\leq \frac{3m}2-1$. 

The following Corollary first appeared in \cite{hoyois2021hermitian} with a proof
at the level of deformations of algebras. It was proved, but not
observed, in an earlier draft of this paper:

\begin{corollary}\label{degen2CW} Let $T$ be $1$-generic and of minimal border rank. Then $T$ degenerates
to $T_{CW,m-1}$.
\end{corollary}

As remarked in a special case of Corollary \ref{degen2CW} in  \cite{homs2021bounds}, Corollary \ref{degen2CW} indicates
that there should be room for improvement with Strassen's laser method.
This  is important for Strassen's laser method, as it says   $1$-degenerate minimal
border rank tensors that are not isomorphic to a Coppersmith-Winograd tensor are subject to
barriers no worse than the Coppersmith-Winograd tensor.
 
\section{Other tensors}\label{examplesect}

We briefly describe the symmetry Lie algebras of other tensors used in the 
laser 
method and a related tensor.
  
  \begin{example} [Strassen's tensor] \label{strassenten} The following is the 
first tensor that was used in the laser method:
  $T_{str,q}=\sum_{j=1}^q a_0\ot b_j\ot
  c_j + a_j\ot b_0\ot c_j\in \BC^{q+1}\ot \BC^{q+1}\ot \BC^q$.
  Then, with blocking $(1,q)\times (1,q)$ in the first two matrices,
  $$
  \tilde\fg_{T_{str,q}}=
  \left\{ \lambda\Id+\begin{pmatrix} 0 & y \\ 0 & X\end{pmatrix},  \mu\Id+ \begin{pmatrix} 0 &
  y \\ 0 & X\end{pmatrix}, \nu\Id+(-X^t)
  \mid X\in \fgl(q), x\in \BC^q, \lambda+\mu+\nu=0\right\}.
  $$
  In particular, $\tdim (\fg_{T_{str,q}})=q^2+q$.    \end{example}
  
  \begin{example}[The small Coppersmith-Winograd tensor]\label{smallcw} Another 
tensor used in 
the laser method is the small Coppersmith-Winograd tensor:
  $T_{cw,q}=\sum_{j=1}^q a_0\ot b_j\ot c_j + a_j\ot b_0\ot c_j+a_j\ot b_j\ot
  c_0\in (\BC^{q+1})^{\ot 3}$.
    Then with blocking  $(1,q )\times
  (1,q )$:
  $$
  \tilde \fg_{T_{cw,q}}=
  \left\{ \left( \begin{pmatrix}-\mu-\nu & 0 \\ 0&\lambda\Id+ X \end{pmatrix},
  \begin{pmatrix}-\lambda-\nu & 0 \\ 0&\mu\Id+  X \end{pmatrix},
  \begin{pmatrix}-\lambda-\mu& 0 \\ 0&\nu\Id+  X \end{pmatrix}\right)
  \mid \lambda,\mu,\nu\in \BC
    X\in \fso(q)\right\}.
  $$
  In particular $\tdim \fg_{T_{cw,q}}=\binom{q}2+1$.  \end{example}

\bibliographystyle{amsplain}

\bibliography{bibSymTensors}

\end{document}